\documentclass[11pt,a4paper]{amsart}

\usepackage[T1]{fontenc}
\usepackage[utf8]{inputenc}
\usepackage{lmodern}
\usepackage{microtype}
\usepackage[a4paper,margin=1in]{geometry}
\usepackage{amsmath,amssymb,mathtools,mathrsfs}
\usepackage{enumitem}
\usepackage{xcolor}
\usepackage{graphicx}
\usepackage{float}
\usepackage{hyperref}
\usepackage[nameinlink,capitalize,noabbrev]{cleveref}

\hypersetup{
  colorlinks=true,
  linkcolor=blue!45!black,
  citecolor=blue!45!black,
  urlcolor=blue!45!black,
  pdftitle={Uniform positive scalar curvature from supercritical quadratic decay},
  pdfauthor={Zhehui Wang and Jintian Zhu}
}

\setlist[enumerate,1]{
  label=(\arabic*),
  leftmargin=2.2em,
  itemsep=0.25em,
  topsep=0.35em
}
\setlist[itemize]{
  leftmargin=2em,
  itemsep=0.2em,
  topsep=0.35em
}

\newtheorem{theorem}{Theorem}[section]

\newtheorem{corollary}[theorem]{Corollary}
\newtheorem{proposition}[theorem]{Proposition}
\newtheorem{lemma}[theorem]{Lemma}

\crefname{conjecture}{conjecture}{conjectures}
\Crefname{conjecture}{Conjecture}{Conjectures}
\crefname{question}{question}{questions}
\Crefname{question}{Question}{Questions}

\newcommand{\Ric}{\operatorname{Ric}}

\numberwithin{equation}{section}
\DeclareMathOperator{\Hess}{Hess}
\DeclareMathOperator{\tr}{tr}
\DeclareMathOperator{\Div}{div}
\DeclareMathOperator{\supp}{supp}
\newcommand{\R}{\mathbb R}

\newcommand{\cB}{\mathcal B}
\newcommand{\cQ}{\mathcal Q}
\newcommand{\cD}{\mathcal D}

\newcommand{\ip}[2]{\langle #1,#2\rangle}
\title[Scalar curvature from quadratic decay]{Uniform positive scalar curvature from supercritical quadratic decay}
\author{Zhehui Wang}
\address{School of Sciences, Great Bay University, Dongguan 523000, China}
\email{wangzhehui@gbu.edu.cn}
\author{Jintian Zhu}
\address{Institute for Theoretical Sciences, Westlake University, Hangzhou 310030, China}
\email{zhujintian@westlake.edu.cn}
\date{}

\begin{document}

\makeatletter
\tagsleft@false
\let\veqno\@@eqno
\makeatother

\begin{abstract}
Let $(M^n,g)$ be a complete connected noncompact smooth Riemannian
manifold without boundary, with $n\ge3$ and positive scalar curvature.
We prove that $M$ admits a complete smooth metric of uniformly positive
scalar curvature if
$$
 \liminf_{d_g(o,x)\to\infty}d_g(o,x)^2R_g(x)>\frac{n-1}{n}
$$
for a fixed point $o\in M$. This extends earlier existence results for orientable manifolds,
due to Chen in dimension three and to the authors in dimensions $4\le n\le7$, to every dimension $n\ge3$ without an orientability assumption.
\end{abstract}
\maketitle

\section{Introduction}\label{sec:intro}

On a noncompact manifold, the scalar curvature of a complete metric can be positive everywhere while tending to zero at infinity. A natural question is how slowly scalar curvature needs to decay at infinity for the underlying manifold to admit
a complete metric with uniformly positive scalar curvature.

Let $(M,g)$ be a complete connected noncompact Riemannian
manifold. Fix $o\in M$ and write $r_o(x)=d_g(o,x)$.
The quantity $r_o^2R_g$ is invariant under constant rescaling of $g$, which makes quadratic decay a natural scale for the above problem. We define the {\it asymptotic quadratic scalar curvature coefficient} by
\begin{equation}\label{eq:Cinfty}
 \mathcal C_\infty(g)
 =\liminf_{r_o(x)\to\infty}r_o(x)^2R_g(x).
\end{equation}
When $R_g>0$, this coefficient is independent of the choice
of $o$ and may equal $+\infty$.
For any $c\ge0$, the condition $\mathcal C_\infty(g)>c$
holds precisely when, for some constant $C>c$,
$$
 R_g(x)\ge \frac{C}{r_o(x)^2}
$$
outside a compact set containing $o$.

Our main theorem gives a sharp sufficient condition for the existence of a complete metric with uniformly positive scalar curvature in every dimension $n\ge 3$.

\begin{theorem}\label{thm:main}
Let $n\ge3$, and let $(M^n,g)$ be a complete connected noncompact
smooth Riemannian manifold without boundary. If
\begin{equation}\label{eq:mainhyp}
 R_g>0\quad\text{on }M,
 \qquad \mathcal C_\infty(g)>\frac{n-1}{n},
\end{equation}
then $M$ admits a complete smooth Riemannian metric $\bar g$ with
$R_{\bar g}\ge1$ everywhere.
\end{theorem}

\subsection{Quadratic decay and earlier results}

Gromov's critical rate of decay conjecture
\cite[Section 3.6.1]{Gromov} predicts a dimensional constant $c_n>0$ with two properties for open orientable $n$-manifolds.
First, every such manifold with a complete metric of positive scalar curvature should admit, for each $0<c<c_n$, a complete metric with $R_g>0$ and $\mathcal C_\infty(g)\ge c$.
Second, the existence of a complete metric with $R_g>0$ and $\mathcal C_\infty(g)>c_n$ should imply the existence of a complete metric of uniformly positive scalar curvature. Chen \cite{Chen} proposed the sharp value $c_n=(n-1)/n$.

The conclusion of Theorem \ref{thm:main} was previously established for orientable three-manifolds by Chen \cite[Corollary 1.6]{Chen}, and for orientable manifolds of dimensions $4\le n\le7$ in our recent work \cite[Theorem 1.3]{WZ}. Theorem \ref{thm:main} proves the second assertion of Gromov's critical rate of decay conjecture with $c_n=(n-1)/n$ in every dimension $n\ge 3$, without an orientability assumption.

The threshold $(n-1)/n$ is sharp, and equality cannot be allowed in \eqref{eq:mainhyp}. Taking the exponent $1/n$ in Chen's warped product construction \cite[Section~2]{Chen} gives a complete metric of positive scalar curvature on $\R^2\times\mathbb T^{n-2}$ with $\mathcal C_\infty(g)=(n-1)/n$. However, this manifold admits no complete metric of uniformly positive scalar curvature \cite[p. 648]{GromovMetric}.

In dimension three, the existence of a complete metric of uniformly positive scalar curvature admits a topological characterization. A connected orientable three-manifold without boundary admits such a metric if and only if it is homeomorphic to a possibly infinite connected sum of spherical three-manifolds and copies of $S^2\times S^1$.
Using the surgery methods of Gromov and Lawson
\cite{GLsurgery} and Schoen and Yau \cite{SYsurgery},
one can construct complete metrics of uniformly positive scalar curvature on these connected sums.
In the converse direction, Chang, Weinberger, and Yu \cite[Theorem~3]{CWY} showed that a connected orientable three-manifold with finitely
generated fundamental group and a complete metric of
uniformly positive scalar curvature has the above
connected sum decomposition.
Bessi\`eres, Besson, and Maillot \cite[Theorem 1.1]{BBM} used Ricci flow to establish the connected sum decomposition for connected orientable three-manifolds admitting a complete metric of bounded geometry and uniformly positive scalar curvature. Gromov \cite[Section 3.10.2]{Gromov} and Wang \cite[Theorem 1.1]{JWang} independently established the general case using $\mu$-bubbles; Wang also uses minimal surfaces to identify the prime factors.

The connection between this classification and quadratic scalar curvature decay was developed by Balacheff, Gil Moreno de Mora Sard\`a, and Sabourau
\cite[Theorem 1.3 and Corollary 1.5]{BGM}.
Using fill-radius estimates, they proved that a complete connected orientable three-manifold with $R_g>0$ and $\mathcal C_\infty(g)>64\pi^2$ has the same connected sum decomposition and therefore admits a complete metric of uniformly positive scalar curvature.
Chen \cite[Theorem 1.5 and Corollary 1.6]{Chen} improved the threshold to the optimal value $2/3$ by constructing a $\mu$-bubble exhaustion.
Moreover, for $3\le n\le7$, Chen \cite[Proposition 1.8]{Chen} proved that every complete orientable noncompact Riemannian manifold $(M^n,g)$ satisfying $R_g>0$ and $\mathcal C_\infty(g)>(n-1)/n$ admits a $\mu$-bubble exhaustion by compact domains whose smooth boundaries admit metrics of positive scalar curvature.

Our recent proof in dimensions $4\le n\le7$ \cite{WZ}
uses Chen's $\mu$-bubble exhaustion to decompose the
manifold into compact blocks. 
A calibration argument establishes positivity of the
relative Yamabe invariant of each block with respect to the induced boundary conformal class.
The extension theorem of Akutagawa and Botvinnik
\cite[Corollary B]{AB} then yields positive scalar curvature
metrics on the blocks with prescribed product collars. Rescaling these metrics and joining adjacent blocks by suitable warped cylinders gives a complete metric with uniformly positive scalar curvature. The upper dimension bound comes from the regularity of the minimizing hypersurfaces used in the exhaustion.

The present proof gives a unified construction in every dimension $n\ge3$, without an orientability assumption. We solve a prescribed weighted mean curvature equation for an entire graph in a weighted product $M\times\R$.
Local gradient estimates and elliptic regularity provide the smoothness needed for the construction in all these dimensions.
Moreover, the graph has a globally defined unit normal even when $M$ is nonorientable.
With a suitable choice of the ambient weight, a conformal change of the induced graph metric produces a complete metric with uniformly positive scalar curvature.

For $n\ge3$, Theorem~\ref{thm:main} shows that $M$ admits a complete metric satisfying \eqref{eq:mainhyp} if and only if it admits a complete metric with uniformly positive
scalar curvature. The reverse implication follows from $\mathcal C_\infty(g)=+\infty$ whenever $\inf_M R_g>0$.
Consequently, obstructions to complete metrics with uniformly positive scalar curvature also rule out complete metrics satisfying \eqref{eq:mainhyp}.
For related obstruction results, see
\cite{GL,Zeidler,Cecchini,SWWZ,Sweeney}.

\subsection{Outline of the proof}
Our construction combines a conformal rescaling of the original metric with an entire graph in a weighted product.
The weighted conformal setup is motivated by the conformal identities and cylindrical blow-ups of Bi and the second author \cite[Sections~2--3]{BZ}.
The tangent profile for the prescribed mean curvature is motivated by Gromov's $\mu$-bubble method \cite[Section 3.7]{Gromov}.
The weighted setting is inspired by the weighted
$\mu$-bubble framework of Chodosh and Li
\cite[Section 3]{CL}.
We work with graphical solutions, whose smoothness follows from local gradient estimates and elliptic regularity, with no upper restriction on the dimension.

For a smooth Riemannian manifold $(N,g_N)$, a function $\theta_N\in C^\infty(N)$, and a real parameter $m>1$, define
\begin{equation}\label{eq:critical-weighted-scalar}
 \mathcal R_m(g_N,\theta_N)
 =R_{g_N}+2\Delta_{g_N}\theta_N
  -\frac{m-2}{m-1}|d\theta_N|_{g_N}^2.
\end{equation}
When $m=\dim N$, the expression
$\mathcal R_n(g_N,\theta_N)$ is related to the scalar curvature of $e^{-2\theta_N/(n-1)}g_N$ by the conformal scalar curvature formula. A uniform positive lower bound for $\mathcal R_n(g_N,\theta_N)$ therefore gives uniformly positive
scalar curvature after this conformal change, provided that $\theta_N$ is bounded below.

Let $(M^n,g)$ satisfy \eqref{eq:mainhyp}, and put $\beta=n-1$.
By the smooth approximation of the distance function, we obtain a smooth proper function
$s\ge1$ with $|ds|_g\le1$ and $s^2R_g\ge C_1>\beta/n$.
Set $h=s^{-2}g$, $\rho=\log s$, and $\eta=-\beta\rho$.
The conformal scalar curvature formula yields
\begin{equation*}\label{eq:intro-background}
 \mathcal R_{n+1}(h,\eta)
 =s^2R_g-\frac\beta n|d\rho|_h^2
 \ge C_1-\frac\beta n=:\delta>0.
\end{equation*}
The hypothesis $\mathcal C_\infty(g)>\beta/n$ allows us to choose $C_1>\beta/n$, yielding the positive lower bound $\delta$.

Since the conformal change $e^{-2\eta/\beta}h$ recovers $g$, we introduce the weight $\Phi_\psi(x,t)=\eta(x)+\psi(t)$
on $(M\times\R,\pi_M^*h+dt^2)$, where $\psi$ is an increasing diffeomorphism of $\R$ with suitably small derivatives.
Fix a sufficiently small constant $\mu>0$, depending only on $n$ and $\delta$, as specified in \eqref{eq:mu-choice}.
For each fixed admissible $\psi$, we construct an entire graph $\Gamma_f(x)=(x,f(x))$ in the band
$|\Phi_\psi|<\beta\pi/(2\mu)$, with prescribed weighted mean curvature $b_\psi=\tan(\mu\Phi_\psi/\beta)$.
The construction combines the method of continuity on closed auxiliary manifolds with local estimates and an exhaustion argument.
Comparison barriers keep the graphs uniformly separated from the band boundary on compact subsets.

Write $G=h+df\otimes df$, $W=\sqrt{1+|df|_h^2}$,
and $\varphi=\Gamma_f^*\Phi_\psi$.
We use the graph weight
\begin{equation*}\label{eq:normalized-graph-weight}
 \theta=\varphi+\log W
       =\psi(f)-\beta\rho+\log W.
\end{equation*}
This choice gives $e^{-\theta}dV_G=e^{-\varphi}dV_h$
and leads to a weighted scalar curvature estimate
independent of the graph slope $|df|_h$.
The estimate in Section \ref{sec:scalar} gives a uniform positive lower bound for $\mathcal R_n(G,\theta)$ provided that $\psi'$ is sufficiently small. Since the band condition also bounds $\theta$ uniformly from below, the conformal metric
$\widetilde g=e^{-2\theta/\beta}G$ has uniformly positive scalar curvature.

To prove completeness, set $u=\psi(f)$. The properness of $\rho$ and the band condition imply that $u$ is proper, and the gradient
estimate gives $\psi'(f)W\le\cB(u)$, where $\cB$ depends only on the fixed background data $(M, h, \rho, \eta)$, $\delta$ and $\mu$.
We can therefore choose $\psi$ before solving the graph equation, with $\psi'$ sufficiently small as a function of $\psi(t)$ to ensure $|du|_{\widetilde g}\le1$ while preserving the conditions required
for existence and curvature estimates.
Properness of $u$, together with the bound $|du|_{\widetilde g}\le1$, implies completeness of
$\widetilde g$. A constant rescaling of $\widetilde g$ then yields a complete metric with scalar curvature at least one, as asserted in Theorem \ref{thm:main}.

\subsection*{Organization of the paper}

Section \ref{sec:scale} constructs the weighted background and introduces the vertical potential.
Section \ref{sec:gradient} establishes the graph identities and a local gradient estimate uniform in the admissible potential.
Section \ref{sec:scalar} proves the weighted scalar
curvature estimate, and Section \ref{sec:existence}
constructs entire solutions of the prescribed weighted mean curvature equation.
Finally, Section \ref{sec:choice} selects the potential and completes the proof of Theorem \ref{thm:main} by establishing completeness and uniformly positive scalar curvature of the resulting conformal metric.

\subsection*{Disclosure on AI assistance.}
The central idea for the construction in this paper was proposed by the authors. ChatGPT 6.0 Astra assisted with exploring intermediate arguments, performing model calculations, and examining proposed proofs for gaps. The authors directed these discussions, critically evaluated the suggestions, and independently verified the arguments and wrote the final manuscript. The authors take full responsibility for its content.

\subsection*{Acknowledgements.} Z. Wang is partially supported by NSFC Grant No.\,12401247 and Guangdong Basic and Applied Basic Research Foundation Grant No.\,{2025A1515140243}. J. Zhu is partially supported by National Key R\&D Program of China 2023YFA1009900, NSFC Grant No.\,12401072, Zhejiang Provincial Natural Science Foundation of China under Grant No.\,LQKWL26A0101 and the start-up fund from Westlake University.

\section{The weighted ambient space}\label{sec:scale}

Fix $(M^n,g)$ as in Theorem~\ref{thm:main}, and set $$\beta=n-1.$$
We first convert the asymptotic scalar curvature bound into a global inequality involving a smooth proper function $s$. Then, we use $s$ to define the horizontal metric and weight, and modify the weight on the product space by adding a vertical potential.

\begin{lemma}\label{lem:scale}
Under the hypotheses of Theorem~\ref{thm:main}, let $C_1$ be a finite
constant with $(n-1)/n<C_1<\mathcal C_\infty(g)$. There is a smooth
proper function $s:M\to[1,\infty)$ such that
\begin{equation}\label{eq:scale}
 |ds|_g\le1,\qquad R_g\ge C_1s^{-2}
 \quad\text{on }M.
\end{equation}
\end{lemma}

\begin{proof}
Choose $\varepsilon>0$ with
$C_1(1+\varepsilon)^2<\mathcal C_\infty(g)$. Applying the smooth approximation theorem for Lipschitz functions
\cite[Theorem~1]{AFMR} to $r_o$, we obtain a smooth function $\rho_0$ satisfying
$$
 |\rho_0-r_o|\le1,\qquad |d\rho_0|_g\le1+\varepsilon.
$$
For a constant $A_0\ge1$ to be chosen, set
$$
 s=\frac{\rho_0+A_0}{1+\varepsilon}.
$$
Then $|ds|_g\le1$ and
$s\ge r_o/(1+\varepsilon)$. The asymptotic hypothesis gives a
compact set $E_0$ containing $o$ in its interior such that
$$
 \frac{C_1}{s^2}
 \le\frac{C_1(1+\varepsilon)^2}{r_o^2}\le R_g
 \quad\text{on }M\setminus E_0.
$$
Note the positive continuous function $R_g$ has a positive minimum in $E_0$. We may therefore choose $A_0$ sufficiently large that
$C_1s^{-2}\le R_g$ on $E_0$ and $s\ge1$ on $M$.
Since $C_1s^{-2}$ decreases as $A_0$ increases, the inequality
on $M\setminus E_0$ remains valid.

Note $r_o\le(1+\varepsilon)s$. Then for every $c\ge1$ the sublevel set
$\{s\le c\}$ is a closed subset of the closed metric ball
$\overline {B_g(o,(1+\varepsilon)c)}$.
This ball is compact by the Hopf--Rinow theorem, so $s$ is proper.
\end{proof}

\begin{lemma}\label{lem:background}
Let $C_1$ and $s$ be as in Lemma~\ref{lem:scale}, and define
\begin{equation}\label{eq:background}
 \begin{gathered}
 \delta=C_1-\frac\beta n>0,\qquad 
 h=s^{-2}g,\qquad \rho=\log s,\qquad \eta=-\beta\rho.
 \end{gathered}
\end{equation}
On $M\times\R$, let $\pi_M:M\times\R\to M$ be the
projection and set
$$
 \widehat h=\pi_M^*h+dt^2,\qquad \widehat\eta=\pi_M^*\eta.
$$
Then $\rho:M\to[0,\infty)$ is proper, and
\begin{equation}\label{eq:background-gradient}
 |d\rho|_h=|ds|_g\le1,\qquad |d\eta|_h\le\beta.
\end{equation}
Moreover,
\begin{equation}\label{eq:cylindrical-weight}
 \begin{aligned}
 \mathcal R_{n+1}(h,\eta)
 &=s^2R_g-\frac{\beta}{n}|d\rho|_h^2\ge\delta,\\
 \mathcal R_{n+1}(\widehat h,\widehat\eta)
 &=\pi_M^*\bigl(\mathcal R_{n+1}(h,\eta)\bigr)\ge\delta.
 \end{aligned}
\end{equation}
\end{lemma}

\begin{proof}
Since $s\ge1$ is proper, $\rho=\log s$ is nonnegative and proper.
By the definitions of $h,\rho,\eta$ and $|ds|_g\le 1$, we know
$$
 |d\rho|_h=s|d\log s|_g=|ds|_g\le1,\qquad
 |d\eta|_h=\beta|d\rho|_h\le\beta.
$$
Then the conformal formulas give
\[
 \begin{split}
 R_h&=s^2\bigl(R_g+2\beta\Delta_g\rho
                    -\beta(\beta-1)|d\rho|_g^2\bigr),\\
 \Delta_h\eta&=-\beta s^2
          \bigl(\Delta_g\rho-(\beta-1)|d\rho|_g^2\bigr),\\
 |d\eta|_h^2&=\beta^2|ds|_g^2.
 \end{split}
\]
Substituting these identities into \eqref{eq:critical-weighted-scalar} gives the first
identity in \eqref{eq:cylindrical-weight}. Its lower bound follows
from $s^2R_g\ge C_1$ and \eqref{eq:background-gradient}.
The second identity follows from the product structure and the
$t$-independence of $\widehat\eta$, which give
$$
 R_{\widehat h}=\pi_M^*R_h,\qquad
 \Delta_{\widehat h}\widehat\eta=\pi_M^*(\Delta_h\eta),\qquad
 |d\widehat\eta|_{\widehat h}^2=\pi_M^*(|d\eta|_h^2).
$$
\end{proof}

For the graph construction, fix $C_1,s$ and the background data
from Lemma~\ref{lem:background}.
For a smooth increasing diffeomorphism $\psi:\R\to\R$, define the
ambient weight
\begin{equation}\label{eq:ambient-potential}
 \Phi_\psi(x,t)=\eta(x)+\psi(t)=\psi(t)-\beta\rho(x).
\end{equation}
The product metric gives
$$
 \Delta_{\widehat h}\Phi_\psi=\Delta_h\eta+\psi''(t),\qquad
 |d\Phi_\psi|_{\widehat h}^2=|d\eta|_h^2+\psi'(t)^2.
$$
Here the horizontal quantities are evaluated at $x$. Substitution in
\eqref{eq:critical-weighted-scalar} yields
\begin{equation}\label{eq:ambient-weight-change}
 \mathcal R_{n+1}(\widehat h,\Phi_\psi)(x,t)
 =\mathcal R_{n+1}(h,\eta)(x)
       +2\psi''(t)-\frac\beta n\psi'(t)^2.
\end{equation}
We will choose $\psi$ so that
$$
 2|\psi''|+\frac{\beta}{n}\psi'^2\le\frac{\delta}{2}
 \qquad\text{on }\R,
$$
which ensures that
$\mathcal R_{n+1}(\widehat h,\Phi_\psi)\ge\delta/2$
throughout $M\times\R$. For each $x\in M$, the map
$t\mapsto\Phi_\psi(x,t)=\eta(x)+\psi(t)$ is an increasing
diffeomorphism of $\R$. Hence every band $\{|\Phi_\psi|<c\}$
with $c>0$ intersects each vertical fiber $\{x\}\times\R$.
We will then construct a smooth entire graph $t=f(x)$ in such a band
for a suitable constant $c>0$.

\section{Weighted graph geometry and a gradient estimate}
\label{sec:gradient}
We first derive graph identities and a local gradient estimate on the
fixed background $(M,h,\eta)$ from Lemma \ref{lem:background}.
The proofs of Lemma~\ref{lem:geometry} and Proposition \ref{prop:gradient}
require only smoothness of the background metric and weight.
With the same hypotheses on the graph, potential, and prescribed function,
both results therefore also apply to a smooth Riemannian manifold
$(N^n,g_N)$ with $\eta_N\in C^\infty(N)$.
The gradient estimate is uniform over the following class of vertical
potentials, allowing $\psi$ to be chosen before solving the graph
equation. Let $\mathscr P$ consist of all smooth increasing
diffeomorphisms $\psi:\R\to\R$ satisfying $\psi(0)=0$ and
\begin{equation}\label{eq:Aclass}
 0<\psi'\le1,\qquad |(\log\psi')'|\le1,\qquad
 (\log\psi')''\ge-1\quad\text{on }\R.
\end{equation}

\subsection{The graph identities}

For the projection $\pi_M:M\times\R\to M$, define the product
metric by
$$
 \widehat h=\pi_M^*h+dt^2.
$$
Fix $\psi\in\mathscr P$ and let
$\Phi_\psi(x,t)=\eta(x)+\psi(t)$. For an open set $U\subset M$ and
$f\in C^\infty(U)$, consider the graph embedding
$$
 \Gamma_f:U\longrightarrow M\times\R,\qquad
 \Gamma_f(x)=(x,f(x)),
$$
and write $\Sigma=\Gamma_f(U)$. We identify $\Sigma$ with $U$
via $\Gamma_f$ and denote the induced metric and restricted weight by
\begin{equation}\label{eq:graphnotation}
 \begin{aligned}
 G&=\Gamma_f^*\widehat h=h+df\otimes df,\\
 \varphi&=\Gamma_f^*\Phi_\psi=\eta+\psi(f).
 \end{aligned}
\end{equation}
In identities on $U$, derivatives of $\psi$ are evaluated at $f$
unless another argument is displayed. We also set
$$
 W=\sqrt{1+|df|_h^2},\qquad V=\frac{\nabla_hf}{W},\qquad \omega=d\log W.
$$
For $X\in T_xU$, write
$$
 \widetilde X=d\Gamma_f(X)
 =X+df(X)\partial_t\in T_{\Gamma_f(x)}\Sigma.
$$
Define the weighted Laplacian on the graph by
$$
 \Delta_{G,\varphi}u=\Delta_Gu-\ip{d\varphi}{du}_G
 \qquad \text{ for } u\in C^\infty(U).
$$

Let $\nabla$ and $\widehat\nabla$ be the Levi-Civita connections of
$h$ and $\widehat h$, respectively. Let $\nu$ be the downward unit normal to $\Sigma$, given by
$$
 \nu(\Gamma_f(x))=(V(x),-W(x)^{-1}).
$$
Define a symmetric $(0,2)$-tensor $A$ on $U$ by
$$
 A_x(X,Y)
 =\ip{\widehat\nabla_{\widetilde X}\nu}{\widetilde Y}_{\widehat h},
 \qquad \text{ for } X,Y\in T_xU.
$$
Thus $A$ is the pullback under $\Gamma_f$ of the second fundamental form of $\Sigma$, with the above sign convention.

For a smooth vector field $Y$ on $U$, set
\[
 \Div_{h,\eta}Y=e^\eta\Div_h(e^{-\eta}Y).
\]
In local coordinates on $U$, write
$$
 f_i=\partial_if,\qquad f^i=h^{ij}f_j.
$$
Let $H$ denote the pullback to $U$ of the mean curvature of $\Sigma$
with respect to $\nu$.
Then
\begin{equation}\label{eq:AH}
 A_{ij}=\frac{(\Hess_hf)_{ij}}W,\qquad
 H=\tr_GA=\Div_hV.
\end{equation}
The normal derivative of the ambient weight and
the corresponding weighted mean curvature $K$, both pulled
back to $U$, are given by
\begin{equation}\label{eq:ambient-mean-curvature}
 \begin{gathered}
 d\Phi_\psi(\nu)\circ \Gamma_f=d\eta(V)-\frac{\psi'(f)}W,\\
 K=H-d\Phi_\psi(\nu)\circ \Gamma_f
   =\Div_{h,\eta}V+\frac{\psi'(f)}W.
 \end{gathered}
\end{equation}

\begin{lemma}\label{lem:geometry}
For every $u\in C^\infty(U)$, the graph Laplacians satisfy
\begin{gather}
 \Delta_Gu=\tr_G\Hess_hu-HV(u),
 \label{eq:lappsi}\\
 \Delta_{G,\varphi}u
 =\tr_G\Hess_hu-KV(u)-\ip{d\eta}{du}_h.
 \label{eq:weighted-lappsi}
\end{gather}
The scalar curvature of $G$ satisfies
\begin{equation}\label{eq:gauss}
 R_G=R_h-2\Ric_h(V,V)+H^2-|A|_G^2.
\end{equation}
Moreover, the weighted Jacobi identity takes the form
\begin{equation}\label{eq:weighted-jacobi}
 \Delta_{G,\varphi}\log W
 =|A|_G^2+(\Ric_h+\Hess_h\eta)(V,V)
     +V(K)+|\omega|_G^2-\psi''(1-W^{-2}).
\end{equation}
\end{lemma}

\begin{proof}
In local coordinates on $U$, we have
\begin{equation}\label{eq:graph-inverse-metric}
 G^{ij}=h^{ij}-\frac{f^if^j}{W^2},\qquad
 \Gamma(G)^k_{ij}-\Gamma(h)^k_{ij}
 =\frac{f^k(\Hess_hf)_{ij}}{W^2}.
\end{equation}
The connection difference gives
$$
 \Delta_Gu
 =\tr_G\Hess_hu
 -\frac{G^{ij}(\Hess_hf)_{ij}}{W^2}\ip{df}{du}_h
 =\tr_G\Hess_hu-HV(u).
$$
This proves \eqref{eq:lappsi}. For $u,v\in C^\infty(U)$,
\eqref{eq:graph-inverse-metric} gives
$$
 \ip{du}{dv}_G=\ip{du}{dv}_h-V(u)V(v).
$$
In particular, $\ip{df}{du}_G=W^{-1}V(u)$, and hence
$$
 \ip{d\varphi}{du}_G
 =\ip{d\eta}{du}_h-d\eta(V)V(u)+\psi'W^{-1} V(u).
$$
By subtracting this drift term from \eqref{eq:lappsi} and using
\eqref{eq:ambient-mean-curvature}, we obtain
\eqref{eq:weighted-lappsi}.

The product structure gives $R_{\widehat h}\circ\Gamma_f=R_h$
and $\Ric_{\widehat h}(\nu,\nu)\circ\Gamma_f=\Ric_h(V,V)$.
Pulling the Gauss equation back to $U$ therefore gives
\eqref{eq:gauss}.

It remains to prove \eqref{eq:weighted-jacobi}.
Consider the vertical translations
$$
 F_\varepsilon:U\longrightarrow M\times\R,\qquad
 F_\varepsilon(x)=(x,f(x)+\varepsilon).
$$
Let $H_\varepsilon$ denote the mean curvature of $F_\varepsilon(U)$
with respect to the downward unit normal, pulled back to $U$ via $F_\varepsilon$. Since $d(f+\varepsilon)=df$,
\eqref{eq:AH} gives $H_\varepsilon=H$.
The initial velocity is $\partial_t$, and at $\Gamma_f(x)$ we have
$$
 \partial_t
 =-W(x)^{-1}\nu(\Gamma_f(x))
 +d\Gamma_f\bigl(W(x)^{-1}V(x)\bigr).
$$
The mean curvature variation formula therefore gives
$$
 0=\left.\frac{dH_\varepsilon}{d\varepsilon}
       \right|_{\varepsilon=0}
  =\Delta_GW^{-1}
    +(|A|_G^2+\Ric_h(V,V))W^{-1}
    +W^{-1}V(H).
$$
Multiplying by $W$ and using
$W\Delta_GW^{-1}=-\Delta_G\log W+|\omega|_G^2$,
we obtain
\begin{equation}\label{eq:jacobi}
 \Delta_G\log W
 =|A|_G^2+\Ric_h(V,V)+V(H)+|\omega|_G^2.
\end{equation}

We next rewrite $V(H)$ in terms of $K$.
Differentiating $\log W$ gives $\omega_i=A_{ij}V^j$, while
\eqref{eq:graph-inverse-metric} yields $\nabla_Gf=W^{-1}V$.
Consequently,
\begin{equation}\label{eq:AV}
 \omega=A(V,\cdot),\qquad
 \ip{\omega}{df}_G=W^{-1}\omega(V),\qquad
 |df|_G^2=1-W^{-2}.
\end{equation}
For every vector field $X$ on $U$, by differentiating
$V=W^{-1}\nabla_hf$ and using \eqref{eq:AV}, we have
$$
 h(\nabla_VV,X)
 =W^{-1}\Hess_hf(V,X)-\omega(V)h(V,X)=\omega(X)-\omega(V)h(V,X).
$$
Thus, by \eqref{eq:graph-inverse-metric},
\begin{equation*}\label{eq:graph-acceleration}
 \nabla_VV
 =\nabla_h\log W-\omega(V)V
 =\nabla_G\log W.
\end{equation*}
It follows that
\begin{equation}\label{eq:graph-weight-derivative}
 V(d\eta(V))
 =\Hess_h\eta(V,V)+d\eta(\nabla_VV)
 =\Hess_h\eta(V,V)+\ip{d\eta}{\omega}_G.
\end{equation}
Differentiating \eqref{eq:ambient-mean-curvature} and using
$V(f)=W(1-W^{-2})$ gives
 \begin{equation}\label{eq:VH}
 V(H)
 =V(K)+V(d\eta(V))-\psi''(1-W^{-2})
 +\psi'W^{-1}\omega(V). 
 \end{equation}
Since $d\varphi=d\eta+\psi' df$, \eqref{eq:AV} implies
$$
 \ip{d\varphi}{\omega}_G
 =\ip{d\eta}{\omega}_G+\psi'W^{-1}\omega(V).
$$
Substituting \eqref{eq:graph-weight-derivative} and \eqref{eq:VH} into \eqref{eq:jacobi}, we obtain
\eqref{eq:weighted-jacobi} by the definition of
$\Delta_{G,\varphi}$.
\end{proof}

For later use, we record the corresponding formulas for $f$
and $\psi(f)$. Taking $u=f$ in \eqref{eq:lappsi} and using
$\tr_G\Hess_hf=WH$ and $V(f)=W-W^{-1}$, we obtain
\begin{equation}\label{eq:lapf}
 \Delta_Gf=W^{-1}H.
\end{equation}
By \eqref{eq:AV} and \eqref{eq:ambient-mean-curvature},
we also have
\begin{equation}\label{eq:weighted-lapf}
 \Delta_{G,\varphi}f
 =W^{-1}(H-d\eta(V))-\psi'(1-W^{-2})
 =W^{-1}K-\psi'.
\end{equation}
Applying the chain rule to $\psi(f)$ and using
\eqref{eq:AV} and \eqref{eq:lapf}, we obtain
\begin{equation}\label{eq:lapz}
 \begin{gathered}
 \Delta_G(\psi(f))
 =\psi'W^{-1}H+\psi''(1-W^{-2}),\\
 |d(\psi(f))|_G^2=\psi'^2(1-W^{-2}).
 \end{gathered}
\end{equation}

\subsection{A uniform gradient estimate}
Let $b\in C^\infty(\mathscr O)$, where $\mathscr O\subset U\times\R$
is an open set containing the graph of $f$. We consider the
prescribed weighted mean curvature equation
\begin{equation}\label{eq:pde}
 K=b(x,f(x)).
\end{equation}
Here $d_xb$ denotes the differential of $b$ in the $x$ variable
with $t$ fixed, and $b_t=\partial_t b$. In formulas on $U$,
$b$ and its derivatives are evaluated at $(x,f(x))$.

Set
$$
 \cD=\psi'(f)W.
$$
The following proposition gives a local bound for $\cD$
under the ratio bounds in \eqref{eq:generalratios}, with
constants independent of $\psi\in\mathscr P$.

\begin{proposition}\label{prop:gradient}
In the setting above, assume that $U\Subset M$.
Let $E\Subset U$ be a nonempty compact set, and fix
$q_0>0$ and $0\le\Lambda<\infty$.
Suppose that $f$ solves \eqref{eq:pde} and that
\begin{equation}\label{eq:generalratios}
 \frac{b_t}{\psi'}\ge q_0,\qquad
 |b|+|d_xb|_h\le\Lambda\frac{b_t}{\psi'}
 \quad\text{on }U.
\end{equation}
Then
\begin{equation}\label{eq:Dbound}
 \sup_E\cD\le C_{E,U},
\end{equation}
where $C_{E,U}<\infty$ depends only on $n$, $q_0$, $\Lambda$, $E$, $U$, and the background data $h$, $\eta$ on $U$, and is independent of $\psi$, $f$, $b$.
\end{proposition}

\begin{proof}
Choose a cutoff function $\zeta\in C_c^\infty(U)$, depending
only on $E$ and $U$, such that $0\le\zeta\le1$ on $U$ and
$\zeta=1$ on $E$. Set
$$
 \Psi=\zeta^2\cD.
$$
Since $\Psi$ has compact support in $U$ and is positive on $E$,
it attains a positive maximum at some $x_0\in U$.
In particular, $\zeta(x_0)>0$.

Near $x_0$, write
$$
 \kappa=d\log\zeta^2,\qquad q=\frac{b_t}{\psi'}.
$$
Differentiating \eqref{eq:pde} along $V$, we obtain
$$
 V(K)=d_xb(V)+q\cD(1-W^{-2}).
$$
By \eqref{eq:AV}, \eqref{eq:weighted-lapf} and
\eqref{eq:weighted-lappsi},
\[
 \begin{aligned}
 \Delta_{G,\varphi}\log\psi'(f)
 &=(\log\psi')''(1-W^{-2})
   +(\log\psi')'\frac KW-\psi'',\\
 \Delta_{G,\varphi}\log\zeta^2
 &=\tr_G\Hess_h\log\zeta^2
   -K\kappa(V)-\ip{d\eta}{\kappa}_h.
 \end{aligned}
\]
All subsequent pointwise calculations are carried out at $x_0$.
Since $\log\Psi$ attains a local maximum there, we have
$d\log\Psi=0$ and $\Delta_G\log\Psi\le0$, and hence
$\Delta_{G,\varphi}\log\Psi\le0$.
Note
$$
 \log\Psi=\log\zeta^2+\log\psi'(f)+\log W.
$$
Combining these identities with
\eqref{eq:weighted-jacobi} gives
\begin{equation}\label{eq:gradientidentity}
 \begin{aligned}
 0\ge\Delta_{G,\varphi}\log\Psi
 ={}&(\Ric_h+\Hess_h\eta)(V,V)
      +|A|_G^2+|\omega|_G^2\\
 &+d_xb(V)+q\cD(1-W^{-2})
      +(\log\psi')'W^{-1}b\\
 &+(\log\psi')''(1-W^{-2})
      -\psi''(2-W^{-2})\\
 &+\tr_G\Hess_h\log\zeta^2
      -b\kappa(V)-\ip{d\eta}{\kappa}_h.
 \end{aligned}
\end{equation}

Since $\overline U$ is compact, we may choose $C_U\ge0$,
depending only on $U$, $h$, $\eta$, such that
$$
 \Ric_h+\Hess_h\eta\ge-C_Uh,\qquad
 |d\eta|_h\le C_U
 \quad\text{on }U.
$$
Since $\psi\in\mathscr P$, we have
$$
 |(\log\psi')'|\le1,\qquad
 |\psi''|=\psi'|(\log\psi')'|\le1,\qquad
 (\log\psi')''\ge-1.
$$
Together with $W\ge1$, these imply
$$
 (\log\psi')''(1-W^{-2})
 -\psi''(2-W^{-2})\ge-3.
$$
Moreover, since $|V|_h\le1$, \eqref{eq:generalratios} yields
$$
 \left|
 d_xb(V)+(\log\psi')'\frac bW-b\kappa(V)
 \right|
 \le\Lambda q(1+|\kappa|_h).
$$
By \eqref{eq:graph-inverse-metric}, the trace term satisfies
$$
 \tr_G\Hess_h\log\zeta^2
 \ge -n|\Hess_h\log\zeta^2|_h.
$$
Using this bound and the preceding estimates in
\eqref{eq:gradientidentity}, and discarding the nonnegative
terms $|A|_G^2$ and $|\omega|_G^2$, we obtain
$$
 q\cD(1-W^{-2})
 \le{}C_U+3+\Lambda q(1+|\kappa|_h)+C_U|\kappa|_h+n|\Hess_h\log\zeta^2|_h.
$$
Dividing by $q\ge q_0$ gives
$$
 \cD(1-W^{-2})
 \le\Lambda+\frac{C_U+3}{q_0}
   +\left(\Lambda+\frac{C_U}{q_0}\right)|\kappa|_h
   +\frac n{q_0}|\Hess_h\log\zeta^2|_h.
$$

The cutoff derivatives satisfy
\[
 \begin{aligned}
 \zeta^2|\kappa|_h
 &=2\zeta|d\zeta|_h,\\
 \zeta^2|\Hess_h\log\zeta^2|_h
 &\le2\zeta|\Hess_h\zeta|_h+2|d\zeta|_h^2.
 \end{aligned}
\]
Multiplying the preceding estimate by $\zeta^2$ and using
$0\le\zeta\le1$, we obtain
\begin{equation}\label{eq:psi-est}
 \Psi(1-W^{-2})\le L_\zeta,
\end{equation}
where
$$
 L_\zeta=\Lambda+\frac{C_U+3}{q_0}
 +2\left(\Lambda+\frac{C_U}{q_0}\right)
       \|d\zeta\|_\infty
 +\frac{2n}{q_0}
   \left(\|\Hess_h\zeta\|_\infty
         +\|d\zeta\|_\infty^2\right).
$$
Here the supremum norms are taken on $U$ with respect to $h$.

If $\cD(x_0)\ge2$, then $\psi'\le1$ gives
$$
 W(x_0)^{-2}
 =\frac{\psi'(f(x_0))^2}{\cD(x_0)^2}
 \le\frac14.
$$
Hence, by \eqref{eq:psi-est} we have $3\Psi(x_0)\le4L_\zeta$.
If $\cD(x_0)<2$, then $\Psi(x_0)\le\cD(x_0)<2$.
Since $\zeta=1$ on $E$, it follows that
$$
 \sup_E\cD
 \le\Psi(x_0)
 \le\max\left\{2,\frac{4}{3}L_\zeta\right\}.
$$
Taking $C_{E,U}=\max\{2,4L_\zeta/3\}$ proves the estimate with the stated dependence.
\end{proof}

For the graph construction, fix the background data
$(M,h,\rho,\eta)$ and $\delta>0$ from Lemma~\ref{lem:background}.
With $\beta=n-1$, choose
\begin{equation}\label{eq:mu-choice}
 0<\mu\le\min\left\{
 \frac{n}{8\beta},
 \frac{\delta}{8(1+2\beta^2n^{-2})}
 \right\}.
\end{equation}
For each $\psi\in\mathscr P$, define
\begin{equation}\label{eq:strip}
 \begin{aligned}
 \mathscr S_\psi
 &=\left\{(x,t)\in M\times\R:
   |\Phi_\psi(x,t)|<\frac{\beta\pi}{2\mu}\right\},\\
 b_\psi(x,t)
 &=\tan\left(\frac{\mu}{\beta}\Phi_\psi(x,t)\right),
 \qquad \text{ for } (x,t)\in\mathscr S_\psi.
 \end{aligned}
\end{equation}
Since $\psi$ is an increasing diffeomorphism of $\R$,
the band $\mathscr S_\psi$ intersects every vertical fiber.
We seek a smooth entire graph contained in $\mathscr S_\psi$
and satisfying
\[
 K=b_\psi(x,f(x)).
\]
The following lemma provides the bounds needed to apply
Proposition~\ref{prop:gradient}.

\begin{lemma}\label{lem:profile}
On $\mathscr S_\psi$, the function $b_\psi$ satisfies
\[
 \begin{gathered}
 \frac{(b_\psi)_t}{\psi'(t)}\ge\frac{\mu}{\beta},\\
 |b_\psi|+|d_xb_\psi|_h
 \le
 \frac{\beta}{\mu}\left(\frac12+\mu\right)
 \frac{(b_\psi)_t}{\psi'(t)}.
 \end{gathered}
\]
For each fixed $x\in M$, the function $t\mapsto b_\psi(x,t)$
is strictly increasing on the interval
$\{t:(x,t)\in\mathscr S_\psi\}$, and tends to $-\infty$ and
$+\infty$ at its lower and upper endpoints, respectively.
\end{lemma}

\begin{proof}
Since $\Phi_\psi(x,t)=-\beta\rho(x)+\psi(t)$,
differentiation gives
\begin{equation}\label{eq:profile-derivatives}
 \begin{aligned}
 (b_\psi)_t
 &=\frac{\mu}{\beta}\psi'(t)(1+b_\psi^2)>0,\\
 d_xb_\psi
 &=-\mu(1+b_\psi^2)d\rho.
 \end{aligned}
\end{equation}
In particular,
\[
 \frac{(b_\psi)_t}{\psi'(t)}
 =\frac{\mu}{\beta}(1+b_\psi^2)
 \ge\frac{\mu}{\beta}.
\]
Using $|d\rho|_h\le1$ from
\eqref{eq:background-gradient} and
$2|b_\psi|\le1+b_\psi^2$, we obtain
$$
 |b_\psi|+|d_xb_\psi|_h
 \le\left(\frac12+\mu\right)(1+b_\psi^2)
 =\frac{\beta}{\mu}\left(\frac12+\mu\right)
   \frac{(b_\psi)_t}{\psi'(t)}.
$$
Finally, for fixed $x$, the quantity
$(\mu/\beta)\Phi_\psi(x,t)$ increases from $-\pi/2$ to
$\pi/2$ along the vertical interval in $\mathscr S_\psi$.
The monotonicity and endpoint limits of $b_\psi$ follow
from those of the tangent function.
\end{proof}

\begin{corollary}\label{cor:uniformD}
Let $U\Subset M$ be open and let $E\Subset U$ be a nonempty
compact set. Suppose that $\psi\in\mathscr P$ and
$f\in C^\infty(U)$ satisfy
$$
 K=b_\psi(x,f(x)),\qquad
 \Gamma_f(U)\subset\mathscr S_\psi.
$$
Then
$$
 \sup_E\cD\le C_{E,U},
$$
where $C_{E,U}<\infty$ depends only on the fixed background data,
$\mu$, $E$ and $U$, and is independent of $\psi$ and $f$. The constant $C_{E,U}$ does not depend on the distance from
the graph to $\partial\mathscr S_\psi$, and no positive lower bound on this distance is required.
\end{corollary}

\begin{proof}
By Lemma~\ref{lem:profile}, the bounds
\eqref{eq:generalratios} hold along the graph with
\[
 q_0=\frac{\mu}{\beta},\qquad
 \Lambda=\frac{\beta}{\mu}\left(\frac12+\mu\right).
\]
The conclusion follows from Proposition~\ref{prop:gradient},
applied with $\mathscr O=\mathscr S_\psi\cap(U\times\R)$.
\end{proof}

\section{Weighted scalar curvature of the graph}\label{sec:scalar}

Throughout this section, we fix the background data of
Lemma \ref{lem:background} and $\mu>0$ satisfying \eqref{eq:mu-choice}. For $\psi\in\mathscr P$ and an open
set $U\subset M$, let $f\in C^\infty(U)$ satisfy
\begin{equation}\label{eq:scalar-graph-assumptions}
 K=b_\psi(x,f(x)),\qquad
 \Gamma_f(U)\subset\mathscr S_\psi.
\end{equation}
We use the graph notation introduced in
Section~\ref{sec:gradient}. Unless otherwise indicated,
derivatives of $\psi$ are evaluated at $f$.

Set
$$
 \theta=\varphi+\log W,\qquad a=d\eta(V).
$$
Since $dV_G=W\,dV_h$, the choice of $\theta$ gives
\begin{equation*}\label{eq:graph-weighted-density}
 e^{-\theta}dV_G=e^{-\varphi}dV_h.
\end{equation*}
Define
\begin{equation}\label{eq:scalar-Q-definition}
 Q=K^2+|A|_G^2-\frac1n\bigl(d\Phi_\psi(\nu)\circ\Gamma_f\bigr)^2+\frac n\beta
   \left|\omega+\frac1n d\varphi\right|_G^2.
\end{equation}

\begin{lemma}\label{lem:scalar-identity}
Equip $\Sigma$ with the induced metric and the modified weight $\theta\circ\Gamma_f^{-1}$. Its weighted scalar curvature,
pulled back to $U$ via $\Gamma_f$, is given by
\begin{equation}\label{eq:scalar-Q-identity}
 \mathcal R_n(G,\theta)
 =\Gamma_f^*\mathcal R_{n+1}(\widehat h,\Phi_\psi)
  +2V(K)+Q-2\psi''.
\end{equation}
\end{lemma}

\begin{proof}
Write
$$
 \begin{aligned}
 \Phi_\nu
 &=d\Phi_\psi(\nu)\circ\Gamma_f
   =a-\frac{\psi'}W,\\
 \Phi_{\nu\nu}
 &=(\Hess_{\widehat h}\Phi_\psi)(\nu,\nu)\circ\Gamma_f=\Hess_h\eta(V,V)+W^{-2}\psi''.
 \end{aligned}
$$
These identities follow from
\eqref{eq:ambient-mean-curvature} and the product structure
of $\widehat h$.

Applying \eqref{eq:lappsi} to $\eta$ and using
\eqref{eq:graph-inverse-metric} and \eqref{eq:lapz}, we obtain
\[
 \Delta_G\varphi
 =\Delta_h\eta+\psi''-\Phi_{\nu\nu}-H\Phi_\nu.
\]
The orthogonal decomposition of the ambient gradient into
tangent and normal components gives
\begin{equation}\label{eq:dphiG}
 |d\varphi|_G^2
 =|d\eta|_h^2+\psi'^2-\Phi_\nu^2.
\end{equation}
Moreover, \eqref{eq:weighted-jacobi} can be written as
$$
 \Delta_{G,\varphi}\log W
 =|A|_G^2+\Ric_h(V,V)+\Phi_{\nu\nu}
  +V(K)+|\omega|_G^2-\psi''.
$$

Using
$$
 \Delta_G\log W
 =\Delta_{G,\varphi}\log W+\ip{d\varphi}{\omega}_G
$$
and $H=K+\Phi_\nu$, we substitute these formulas and
\eqref{eq:gauss} into the definition
\eqref{eq:critical-weighted-scalar} to obtain
$$
 \mathcal R_n(G,\theta)
 =R_h+2\Delta_h\eta
     +K^2-\Phi_\nu^2+|A|_G^2+2V(K)+\frac n\beta|\omega|_G^2
   +\frac2\beta\ip{\omega}{d\varphi}_G
   -\frac{n-2}{\beta}|d\varphi|_G^2.
$$
After applying \eqref{eq:ambient-weight-change} and \eqref{eq:dphiG}, this gives
$$
 \mathcal R_n(G,\theta)
 =\Gamma_f^*\mathcal R_{n+1}(\widehat h,\Phi_\psi)
      +2V(K)-2\psi''
 +K^2+|A|_G^2-\frac{\Phi_\nu^2}{n}
 +\frac n\beta|\omega|_G^2
   +\frac2\beta\ip{\omega}{d\varphi}_G
   +\frac1{n\beta}|d\varphi|_G^2.
$$
The last three terms combine as
\[
 \frac n\beta|\omega|_G^2
 +\frac2\beta\ip{\omega}{d\varphi}_G
 +\frac1{n\beta}|d\varphi|_G^2
 =
 \frac n\beta\left|\omega+\frac1n d\varphi\right|_G^2.
\]
This proves \eqref{eq:scalar-Q-identity}.
\end{proof}

The band condition $\Gamma_f(U)\subset\mathscr S_\psi$ and $W\ge1$ also give
\begin{equation}\label{eq:normalized-weight-lower}
 |\theta-\log W|<\frac{\beta\pi}{2\mu},
 \qquad
 \theta\ge-\frac{\beta\pi}{2\mu}.
\end{equation}
We now establish a lower bound for the weighted scalar curvature.

\begin{proposition}\label{prop:scalar}
Let $\psi\in\mathscr P$ and let $f\in C^\infty(U)$ satisfy
\eqref{eq:scalar-graph-assumptions}, where $b_\psi$ and
$\mathscr S_\psi$ are defined in \eqref{eq:strip}.
Set
\begin{equation*}\label{eq:scalar-constant}
 C_n=\frac{n^2+1}{n(n+1)}.
\end{equation*}
Then
\begin{equation}\label{eq:scalarbound}
 \mathcal R_n(G,\theta)
 \ge\frac{3\delta}{4}-C_n\psi'^2
 \qquad\text{on }U.
\end{equation}
\end{proposition}

To prove Proposition~\ref{prop:scalar}, set $\alpha=W^{-2}$
and introduce the nonnegative auxiliary quantity
\begin{equation}\label{eq:scalar-Qalpha}
 Q_\alpha=
 \begin{cases}
 \displaystyle
 \frac{n+1}{n}\alpha K^2
 +\frac{\bigl(a+n(1-\alpha)K\bigr)^2}
        {n\beta(1-\alpha)},
 &0<\alpha<1,\\[6pt]
 \displaystyle\frac{n+1}{n}K^2,
 &\alpha=1.
 \end{cases}
\end{equation}
We next establish two lemmas. Lemma \ref{lem:scalar-Q} gives algebraic lower bounds for $Q$ and $Q_\alpha$. Lemma \ref{lem:scalar-horizontal} combines these estimates with the prescribed weighted mean curvature equation to bound $2V(K)+\frac12Q_\alpha$ from below.
The remaining half of $Q_\alpha$ will be used to absorb the mixed term involving $\psi'K$ in the proof of Proposition \ref{prop:scalar}.

\begin{lemma}\label{lem:scalar-Q}
On $U$, we have
\begin{equation}\label{eq:scalar-Q-lower}
 Q\ge Q_\alpha-\frac{2\psi'\sqrt\alpha}{n}K
\end{equation}
and
\begin{equation}\label{eq:scalar-Qalpha-lower}
 Q_\alpha\ge\frac{n}{2\beta}K^2-\frac{\beta}{n}.
\end{equation}
\end{lemma}

\begin{proof}
We first prove \eqref{eq:scalar-Q-lower}.
At a point where $df\ne0$, choose a $G$-orthonormal basis
$e_1,\ldots,e_n$ with $e_1$ in the direction of $\nabla_Gf$,
and set
$$
 \xi=A(e_1,e_1).
$$
Since $\nabla_Gf=W^{-1}V$ and $|df|_G^2=1-\alpha$,
\begin{equation}\label{eq:Ve1}
 V=\sqrt{\frac{1-\alpha}{\alpha}}\,e_1,
 \qquad
 df(e_1)=\sqrt{1-\alpha}.
\end{equation}
Note
$$
 H-\xi=\sum_{i=2}^n A(e_i,e_i).
$$
Applying the Cauchy--Schwarz inequality to these $\beta=n-1$ diagonal
entries, we obtain
$$
 \begin{aligned}
 |A|_G^2
 =\sum_{i,j=1}^n A(e_i,e_j)^2&\ge\xi^2+\sum_{i=2}^n A(e_i,e_i)^2\\
 &\ge\xi^2+\frac1\beta
       \left(\sum_{i=2}^n A(e_i,e_i)\right)^2=\xi^2+\frac{(H-\xi)^2}{\beta}.
 \end{aligned}
$$
Substituting $H=K+a-\psi'\sqrt\alpha$, we obtain
\begin{equation}\label{eq:scalar-diagonal-data}
|A|_G^2
 \ge\xi^2+\frac{(K+a-\psi'\sqrt\alpha-\xi)^2}{\beta}.
\end{equation}
By $\omega=A(V,\cdot)$ and \eqref{eq:Ve1}, we have
\begin{equation}\label{eq:scalar-omega-component}
 \omega(e_1)=\sqrt{\frac{1-\alpha}{\alpha}}\,\xi.
\end{equation}
Moreover, $a=d\eta(V)$ implies
$$
 d\eta(e_1)=\sqrt{\frac{\alpha}{1-\alpha}}\,a.
$$
Using $d\varphi=d\eta+\psi' df$ and
$df(e_1)=\sqrt{1-\alpha}$, we get
\begin{equation}\label{eq:scalar-weight-component}
 d\varphi(e_1)
 =\sqrt{\frac{\alpha}{1-\alpha}}\,a
  +\psi'\sqrt{1-\alpha}.
\end{equation}

Since the basis is $G$-orthonormal,
$$
 \left|\omega+\frac1n d\varphi\right|_G^2
 \ge
 \left(\omega(e_1)+\frac1n d\varphi(e_1)\right)^2.
$$
Using \eqref{eq:scalar-diagonal-data},
\eqref{eq:scalar-omega-component}, and
\eqref{eq:scalar-weight-component} in
\eqref{eq:scalar-Q-definition}, we obtain
\[
 \begin{aligned}
 Q\ge{}&
 K^2+\xi^2
 +\frac{(K+a-\psi'\sqrt\alpha-\xi)^2}{\beta}
 -\frac{(a-\psi'\sqrt\alpha)^2}{n}\\
 &+\frac n\beta
 \left(
  \sqrt{\frac{1-\alpha}{\alpha}}\,\xi
  +\frac{a}{n}\sqrt{\frac{\alpha}{1-\alpha}}
  +\frac{\psi'}n\sqrt{1-\alpha}
 \right)^2.
 \end{aligned}
\]
Expanding the right-hand side as a quadratic expression
in $\xi$ gives
$$
 Q\ge
 \frac{n}{\beta\alpha}\xi^2
 -\frac2\beta\left(K-\frac{\psi'}{\sqrt\alpha}\right)\xi
 +\frac n\beta K^2+\frac{2Ka}{\beta}
 -\frac{2\psi'\sqrt\alpha}{\beta}K
 +\frac{a^2}{n\beta(1-\alpha)}
 +\frac{\psi'^2}{n\beta}.
$$
Completing the square and using
\eqref{eq:scalar-Qalpha}, we get
$$
 Q\ge
 \frac{n}{\beta\alpha}
 \left(
  \xi-\frac{\alpha K}{n}
      +\frac{\psi'\sqrt\alpha}{n}
 \right)^2
 +Q_\alpha-\frac{2\psi'\sqrt\alpha}{n}K.
$$
Discarding the square proves
\eqref{eq:scalar-Q-lower} wherever $df\ne0$.

At a critical point of $f$, we have
$$
 \alpha=1,\qquad V=\omega=a=0,\qquad H=K-\psi'.
$$
Using $|A|_G^2\ge H^2/n$ and discarding the nonnegative
last term in \eqref{eq:scalar-Q-definition}, we obtain
$$
 \begin{aligned}
 Q
 &\ge K^2+\frac{(K-\psi')^2}{n}-\frac{\psi'^2}{n}\\
 &=\frac{n+1}{n}K^2-\frac{2\psi'}nK=Q_1-\frac{2\psi'}nK.
 \end{aligned}
$$
Thus \eqref{eq:scalar-Q-lower} holds on $U$.

We next prove \eqref{eq:scalar-Qalpha-lower}.
For $0<\alpha<1$, the bounds
$|d\eta|_h\le\beta$ and $|V|_h^2=1-\alpha$ imply
$$
 a^2\le\beta^2(1-\alpha).
$$
By Young's inequality, we have
$$
 \frac{2|Ka|}{\beta}
 \le
 \frac{n(1-\alpha)}{2\beta}K^2
 +\frac{2a^2}{n\beta(1-\alpha)}.
$$
Expanding \eqref{eq:scalar-Qalpha} and applying these bounds,
we obtain
$$
 \begin{aligned}
 Q_\alpha
 &=\frac{n^2-\alpha}{n\beta}K^2
   +\frac{2Ka}{\beta}
   +\frac{a^2}{n\beta(1-\alpha)}\\
 &\ge
 \left(
  \frac n{2\beta}
  +\frac{\alpha(n^2-2)}{2n\beta}
 \right)K^2
 -\frac{a^2}{n\beta(1-\alpha)}\\
 &\ge\frac{n}{2\beta}K^2-\frac{\beta}{n}.
 \end{aligned}
$$
When $\alpha=1$, the same estimate follows from
$$
 Q_1=\frac{n+1}{n}K^2,
 \qquad
 \frac{n+1}{n}\ge\frac{n}{2\beta}.
$$
This proves \eqref{eq:scalar-Qalpha-lower} and completes the proof.
\end{proof}

We now combine \eqref{eq:scalar-Qalpha-lower} with
the prescribed weighted mean curvature equation.

\begin{lemma}\label{lem:scalar-horizontal}
On $U$, we have
\begin{equation}\label{eq:scalar-half-Q}
 2V(K)+\frac12Q_\alpha\ge-\frac{\delta}{4}.
\end{equation}
\end{lemma}

\begin{proof}
Differentiating $K=b_\psi(x,f(x))$ along $V$ and using
\eqref{eq:profile-derivatives} together with
$V(f)=W(1-\alpha)$, we obtain
\begin{equation*}\label{eq:scalar-profile-derivative}
 2V(K)
 =2(1+K^2)
 \left(
  \frac{\mu}{\beta}\psi'W(1-\alpha)-\mu V(\rho)
 \right).
\end{equation*}
Since $\psi'>0$ and
$$
 |V(\rho)|\le|d\rho|_h|V|_h\le1,
$$
it follows that
$$
 2V(K)\ge-2\mu(1+K^2).
$$
On the other hand, by \eqref{eq:mu-choice} we know
$4\mu\beta/n\le1/2$. By $Q_\alpha\ge0$ and
\eqref{eq:scalar-Qalpha-lower}, we have
$$
 \frac12Q_\alpha
 \ge\frac{4\mu\beta}{n}Q_\alpha
 \ge2\mu K^2-\frac{4\mu\beta^2}{n^2}.
$$
Adding these inequalities and using \eqref{eq:mu-choice}
once more, we conclude that
$$
 2V(K)+\frac12Q_\alpha
 \ge-2\mu\left(1+\frac{2\beta^2}{n^2}\right)
 \ge-\frac{\delta}{4}.
$$
\end{proof}
Now we are ready to prove Proposition \ref{prop:scalar}.
\begin{proof}[Proof of Proposition \ref{prop:scalar}]
For every $0<\alpha\le1$, by \eqref{eq:scalar-Qalpha}, we have
$$
 \begin{aligned}
 \frac12Q_\alpha-\frac{2\psi'\sqrt\alpha}{n}K
 &\ge
 \frac{n+1}{2n}\alpha K^2
 -\frac{2\psi'\sqrt\alpha}{n}K\\
 &=
 \frac{n+1}{2n}
 \left(\sqrt\alpha K-\frac{2\psi'}{n+1}\right)^2
 -\frac{2\psi'^2}{n(n+1)}\\
 &\ge-\frac{2\psi'^2}{n(n+1)}.
 \end{aligned}
$$
Combining this estimate with
\eqref{eq:scalar-Q-lower} and \eqref{eq:scalar-half-Q},
we get
\begin{equation}\label{eq:scalar-before-absorption}
 2V(K)+Q
 \ge-\frac{\delta}{4}
     -\frac{2\psi'^2}{n(n+1)}.
\end{equation}

Pulling back \eqref{eq:ambient-weight-change} via $\Gamma_f$
and substituting the result into \eqref{eq:scalar-Q-identity},
we obtain
$$
 \mathcal R_n(G,\theta)
 =\mathcal R_{n+1}(h,\eta)
  +2V(K)+Q-\frac{\beta}{n}\psi'^2.
$$
Using $\mathcal R_{n+1}(h,\eta)\ge\delta$ from
Lemma~\ref{lem:background},
\eqref{eq:scalar-before-absorption}, and $\beta=n-1$, we have
$$
 \begin{aligned}
 \mathcal R_n(G,\theta)
 &\ge\frac{3\delta}{4}
   -\left(\frac{\beta}{n}
          +\frac{2}{n(n+1)}\right)\psi'^2\\
 &=\frac{3\delta}{4}
   -\frac{n^2+1}{n(n+1)}\psi'^2.
 \end{aligned}
$$
This proves \eqref{eq:scalarbound}.
\end{proof}

\section{Entire graphs in the ambient band}
\label{sec:existence}
We first establish local regularity and existence on closed Riemannian
manifolds with arbitrary smooth background data. We then apply these
results to the fixed background data $(M,h,\rho,\eta)$ and  the associated
constant $\delta>0$ from
Lemma \ref{lem:background}, with $\mu$ chosen as in
\eqref{eq:mu-choice}, to construct an entire graph on $M$.

Let $(N^n,g_N)$ be a smooth Riemannian manifold, with $n\ge3$,
and let $\eta_N\in C^\infty(N)$. Fix $\mu>0$ and $\psi\in\mathscr P$ satisfying
\begin{equation}\label{eq:exist-potential-monotonicity}
 |(\log\psi')'|\le\frac{\mu}{2\beta}
 \qquad\text{on }\R.
\end{equation}
Let $\pi_N:N\times\R\to N$ be the projection, and equip
$N\times\R$ with the product metric $\pi_N^*g_N+dt^2$. Set
$$ \Phi_{\psi,N}(x,t)=\eta_N(x)+\psi(t),$$
and define
$$\mathscr S_{\psi,N}=\left\{(x,t)\in N\times\R:
 |\Phi_{\psi,N}(x,t)|<\frac{\beta\pi}{2\mu}\right\}.$$
On this band, set
$$b_{\psi,N}(x,t)=\tan\left(\frac{\mu}{\beta}\Phi_{\psi,N}(x,t)\right).$$
For a graph function $u$, let $\mathcal K_N(u)$ denote the weighted
mean curvature of its graph with respect to the downward unit
normal. Explicitly,
\begin{equation}\label{eq:exist-curvature-operator}
 \mathcal K_N(u)
 =\Div_{g_N,\eta_N}\left(
   \frac{\nabla_{g_N}u}{\sqrt{1+|du|_{g_N}^2}}\right)
  +\frac{\psi'(u)}{\sqrt{1+|du|_{g_N}^2}}.
\end{equation}
By \eqref{eq:exist-potential-monotonicity},  for every $(x,t)\in\mathscr S_{\psi,N}$ and $\xi\in T_x^*N$ we have
\begin{equation}\label{eq:exist-effective-monotonicity}
 \begin{aligned}
 \partial_t\left(
 b_{\psi,N}(x,t)-\frac{\psi'(t)}{\sqrt{1+|\xi|_{g_N}^2}}
 \right)&=
 \frac{\mu}{\beta}\psi'(t)\bigl(1+b_{\psi,N}(x,t)^2\bigr)
 -\frac{\psi''(t)}{\sqrt{1+|\xi|_{g_N}^2}}\\
 &\ge\frac{\mu}{2\beta}\psi'(t)>0.
 \end{aligned}
\end{equation}
Norms of functions on $N$ and $N\times\R$ are computed using
$g_N$ and $\pi_N^*g_N+dt^2$, respectively.

The following lemma provides the local regularity and compactness needed in the existence argument.
\begin{lemma}\label{lem:elliptic}
Let $U\Subset N$ be a coordinate domain, and fix $0<\tau<1$.
Let $\mathscr O\subset N\times\R$ be open, and suppose that
$b_m\in C^\infty(\mathscr O)$ and
$f_m\in C^{2,\tau}_{\mathrm{loc}}(U)$ satisfy
\begin{equation}\label{eq:elliptic-pde}
 \mathcal K_N(f_m)=b_m(x,f_m(x))
 \qquad\text{on }U.
\end{equation}
Assume that all graphs $\Gamma_{f_m}(U)$ lie in a fixed compact
set $\mathscr E\Subset\mathscr O$, 
$\sup_m\|df_m\|_{C^0(U,g_N)}<\infty$, and  $b_m$ and all
their derivatives are uniformly bounded on a fixed open neighborhood
$\mathscr O_0\Subset\mathscr O$ of $\mathscr E$. Then, every $f_m$ is smooth, and for every open set $U'\Subset U$
and integer $\ell\ge0$,
$$
 \sup_m\|f_m\|_{C^\ell(U')}\le C_\ell<\infty.
$$
Here $C_\ell$ depends only on $\ell$, $U'$, $U$, $g_N$, $\eta_N$, $\psi$,
$\mathscr E$, $\mathscr O_0$,
$\sup_m\|df_m\|_{C^0(U, g_N)}$, and
$\sup_m\|b_m\|_{C^{\ell+1}(\mathscr O_0)}$.
In particular, $\{f_m\}$ has a subsequence converging in
$C^\infty_{\mathrm{loc}}(U)$ to a smooth function.
\end{lemma}

\begin{proof}
Since all graphs lie in the compact set $\mathscr E$,
$$
 \sup_m\|f_m\|_{C^0(U)}
 \le\max_{(x,t)\in\mathscr E}|t|<\infty.
$$
Together with the assumed gradient bound, this gives a constant
$L<\infty$ such that
$$
 \|f_m\|_{C^0(U)}+\|df_m\|_{C^0(U, g_N)}\le L
 \qquad\text{for every }m.
$$

Fix $U'\Subset U$ and choose open sets
$U'\Subset U_1\Subset U_0\Subset U$.
Let $D$ denote differentiation in the fixed coordinates on $U$. On $\overline{U_0}$, the $g_N$-norm and the
Euclidean norm are uniformly equivalent.

Choose
$\zeta\in C_c^\infty(\mathscr O_0)$ with $0\le\zeta\le1$,
equal to one near $\mathscr E$.
Let $\widetilde b_m\in C^\infty(N\times\R)$ be the extension
of $\zeta b_m$ by zero.
Since $\widetilde b_m=b_m$ near $\mathscr E$,
\eqref{eq:elliptic-pde} holds with $\widetilde b_m$ in place of $b_m$.
For every integer $k\ge0$,
\begin{equation}\label{eq:tildebm}
 \sup_m\|\widetilde b_m\|_{C^k(N\times\R)}
 \le C_k\sup_m\|b_m\|_{C^k(\mathscr O_0)},
\end{equation}
where $C_k$ depends only on $k$ and the fixed cutoff function.

Let $m_{\eta_N}=e^{-\eta_N}\sqrt{\det((g_N)_{ij})}$ be the coordinate
density of $e^{-\eta_N}dV_{g_N}$.
For independent variables $t\in\R$ and
$\xi=(\xi_1,\ldots,\xi_n)\in\R^n$, define
$$
 J^i(x,\xi)
 =\frac{m_{\eta_N}(x)g_N^{ik}(x)\xi_k}
       {\sqrt{1+g_N^{pq}(x)\xi_p\xi_q}}
$$
and
$$
T_m(x,t,\xi)
 =m_{\eta_N}(x)\left(
 \widetilde b_m(x,t)-\frac{\psi'(t)}
 {\sqrt{1+g_N^{pq}(x)\xi_p\xi_q}}\right).
$$
The coordinate expression for $\Div_{g_N,\eta_N}$ turns
\eqref{eq:elliptic-pde} into
\begin{equation}\label{eq:elliptic-divergence}
 \partial_iJ^i(x,Df_m)-T_m(x,f_m,Df_m)=0.
\end{equation}
Set $W_m=\sqrt{1+|df_m|_{g_N}^2}$ and write
$f_m^i=g_N^{ij}\partial_jf_m$.
Differentiation with respect to $\xi_k$ gives
\begin{equation}\label{eq:exist-principal}
 C_m^{ik}
 :=\frac{\partial J^i}{\partial\xi_k}(x,Df_m)
 =\frac{m_{\eta_N}}{W_m}
  \left(g_N^{ik}-\frac{f_m^if_m^k}{W_m^2}\right).
\end{equation}
For every $\xi\in\R^n$,
$$
  \frac{m_{\eta_N}}{W_m^3}|\xi|_{g_N}^2
 \le C_m^{ik}\xi_i\xi_k
 \le\frac{m_{\eta_N}}{W_m}|\xi|_{g_N}^2.
$$
Expanding \eqref{eq:elliptic-divergence}, we obtain
\begin{equation}\label{eq:elliptic-nondivergence}
 C_m^{ik}\partial_i\partial_kf_m
 =T_m(x,f_m,Df_m)-\sum_iJ^i_{x^i}(x,Df_m).
\end{equation}
Here $J^i_{x^i}$ denotes differentiation with respect to $x^i$
with $\xi$ fixed. Thus $(C_m^{ik})$ is the principal coefficient
matrix of \eqref{eq:elliptic-divergence}.
Since $1\le W_m\le\sqrt{1+L^2}$ and $m_{\eta_N}$ is bounded above
and away from zero on $\overline{U_0}$, the preceding inequalities
and the uniform equivalence of the $g_N$-norms and Euclidean norms show
that the equation is uniformly elliptic on $U_0$, with ellipticity
constants independent of $m$.

The bounds for $f_m$ and $Df_m$ place $(f_m,Df_m)$ in a fixed
bounded subset of $\R\times\R^n$. For $x\in U_0$ and $(t,\xi)$ in this
set, the smoothness of $g_N$, $\eta_N$, and $\psi$, together with
\eqref{eq:tildebm}, gives uniform bounds for $J^i$, $T_m$, and their
derivatives. These bounds and the uniform ellipticity of
$C_m^{ik}$ allow us to apply the interior H\"older estimate of the gradient \cite[Theorem 13.1]{GT}, which gives
$$
 \sup_m\|f_m\|_{C^{1,\tau_0}(U_1)}\le C
$$
for some $\tau_0\in(0,1)$ and $C<\infty$, both independent of $m$.

The coefficients $C_m^{ik}$ and the right hand side of
\eqref{eq:elliptic-nondivergence} are consequently uniformly
bounded in $C^{0,\tau_0}(U_1)$.
Interior Schauder estimates give uniform $C^{2,\tau_0}$ bounds
for $f_m$ on relatively compact subdomains of $U_1$.
These $C^{2,\tau_0}$ estimates give uniform $C^{1,\tau_0}$
bounds for the coefficients and the right hand side of
\eqref{eq:elliptic-nondivergence} on the same subdomains.
A standard bootstrap argument on nested subdomains containing $\overline{U'}$ then yields, for every integer $\ell\ge0$,
$$
 \sup_m\|f_m\|_{C^\ell(U')}\le C_\ell.
$$
Together with \eqref{eq:tildebm}, the
preceding elliptic estimates show that $C_\ell$ depends only on
$\ell$, $U'$, $U$, $\mathscr E$, $\mathscr O_0$, the fixed data $g_N$, $\eta_N$, $\psi$,
and the bounds $\sup_m\|df_m\|_{C^0(U, g_N)}$ and
$\sup_m\|b_m\|_{C^{\ell+1}(\mathscr O_0)}$,
and is independent of $m$.

Finally, since $U'$ and $\ell$ are arbitrary, every $f_m$ is smooth.
By the Arzel\`a--Ascoli theorem and a diagonal argument,
$\{f_m\}$ has a subsequence converging in
$C^\infty_{\mathrm{loc}}(U)$ to a smooth function.
\end{proof}

We next solve the prescribed weighted mean curvature equation on a closed Riemannian manifold.
\begin{proposition}\label{prop:closed}
Let $(N,g_N)$ be a closed smooth Riemannian manifold of dimension $n\ge 3$,
and let $\eta_N\in C^\infty(N)$.
Fix $\mu>0$ and $\psi\in\mathscr P$ satisfying
\eqref{eq:exist-potential-monotonicity}, and define
$\Phi_{\psi,N}$, $\mathscr S_{\psi,N}$, $b_{\psi,N}$, and $\mathcal K_N$ as above.
Then there exists $f\in C^\infty(N)$ satisfying
\begin{equation}\label{eq:exactpde}
\mathcal K_N(f)=b_{\psi,N}(x,f(x))
 \qquad\text{on }N
\end{equation}
and
\begin{equation}\label{eq:exist-strip}
 \Gamma_f(N)\subset\mathscr S_{\psi,N}.
\end{equation}
\end{proposition}

\begin{proof}
Fix $0<\tau<1$ and set $f_0=\psi^{-1}(-\eta_N).$
Then $\Phi_{\psi,N}(x,f_0(x))=0$ and $b_{\psi,N}(x,f_0(x))=0$.
For $\lambda\in\R$, define
$$
 b_{\lambda,N}(x,t)
 =b_{\psi,N}(x,t)+(1-\lambda)\mathcal K_N(f_0)(x)
 \qquad\text{on }\mathscr S_{\psi,N}.
$$
Consider
$$
\mathcal U
 =\{u\in C^{2,\tau}(N):
           \Gamma_u(N)\subset\mathscr S_{\psi,N}\}.
$$
Since $N$ is compact, $\mathcal U$ is open in $C^{2,\tau}(N)$.
Define $\cQ_\lambda:\mathcal U\to C^{0,\tau}(N)$ by
\begin{equation}\label{eq:closed-homotopy}
 \cQ_\lambda(u)=\mathcal K_N(u)-b_{\lambda,N}(x,u(x)),
\end{equation}
and let
$$
 \mathcal I
 =\{\lambda\in[0,1]:
       \cQ_\lambda(u)=0\text{ for some }u\in\mathcal U\}.
$$
We have $f_0\in\mathcal U$ and $\cQ_0(f_0)=0$, so
$0\in\mathcal I$.

We first establish a priori estimates for all solutions of
$\cQ_\lambda(u)=0$, with constants independent of $\lambda$.
For $0<\varepsilon\le\pi/4$, set
\begin{equation}\label{eq:closedbarriers}
 f_\pm
 =\psi^{-1}\left(
    -\eta_N\pm\frac{\beta}{\mu}(\frac\pi 2-\varepsilon)
   \right).
\end{equation}
Since $\eta_N$ is bounded on $N$, the values
$-\eta_N(x)\pm\frac{\beta}{\mu}(\frac\pi 2-\varepsilon)$
lie in a fixed compact interval for all $x\in N$ and $0<\varepsilon\le\pi/4$.
Since $\eta_N$ and $\psi^{-1}$ are smooth and $N$ is compact,
the functions $f_\pm$ have $C^2(N)$ bounds independent of
$\varepsilon$. It follows from
\eqref{eq:exist-curvature-operator} that
$$
 \max\{\|\mathcal K_N(f_-)\|_{C^0(N)},
        \|\mathcal K_N(f_+)\|_{C^0(N)}\}
 \le C_{\mathrm{bar}},
$$
where $C_{\mathrm{bar}}$ is independent of $\varepsilon$.
Choose and fix $\varepsilon>0$ sufficiently small that
$$
 \cot\varepsilon>
 C_{\mathrm{bar}}+\|\mathcal K_N(f_0)\|_{C^0(N)}+1.
$$
Since $b_{\psi, N}(x,f_\pm(x))=\pm\cot\varepsilon$, we obtain
\begin{equation}\label{eq:fpm}
 \cQ_\lambda(f_-)>0,\qquad
 \cQ_\lambda(f_+)<0
 \qquad\text{on }N
\end{equation}
for every $\lambda\in[0,1]$.

Because $b_{\lambda,N}-b_{\psi,N}$ is independent of $t$,
\eqref{eq:exist-effective-monotonicity} holds with $b_{\lambda,N}$
in place of $b_{\psi,N}$.
Fix $\lambda\in[0,1]$ and let $u\in\mathcal U$ satisfy
$\cQ_\lambda(u)=0$.
Suppose that $u\ge f_+$ at some point of $N$.
By compactness of $N$, the function $u-f_+$ attains a
nonnegative maximum at some $x_+\in N$.
At $x_+$, we have $u\ge f_+$, $du=df_+$, and
$\Hess_{g_N}u\le\Hess_{g_N}f_+$.
Ellipticity of $\cQ_\lambda$, \eqref{eq:exist-effective-monotonicity} and \eqref{eq:fpm} give
$$
 0=\cQ_\lambda(u)(x_+)
 \le\cQ_\lambda(f_+)(x_+)<0,
$$
a contradiction. Hence $u<f_+$ on $N$. Similarly, $u>f_-$ on $N$. We then conclude that every solution $u\in\mathcal U$ satisfies
\begin{equation}\label{eq:exist-trapping}
 f_-<u<f_+
 \qquad\text{on }N.
\end{equation}

Set
$$
 \mathscr E
 =\{(x,t)\in N\times\R:f_-(x)\le t\le f_+(x)\}.
$$
By \eqref{eq:closedbarriers}, $\mathscr E$ is a compact subset
of $\mathscr S_{\psi,N}$, and \eqref{eq:exist-trapping} gives
$\Gamma_u(N)\subset\mathscr E$ for every solution $u\in\mathcal U$.
Since $b_{\lambda,N}=b_{\psi,N}+(1-\lambda)\mathcal K_N(f_0)$,
the functions $b_{\lambda,N}$ and all their derivatives in $x$ and $t$
are uniformly bounded for $\lambda\in[0,1]$ on a fixed
neighborhood of $\mathscr E$ compactly contained in
$\mathscr S_{\psi, N}$.

Choose finitely many coordinate domains
$U_i'\Subset U_i\Subset N$ such that $N=\bigcup_iU_i'$.
Each solution $u\in\mathcal U$ has a finite $C^1(N)$ norm.
Lemma~\ref{lem:elliptic}, applied on these charts, therefore
shows that every such solution is smooth.
On $\mathscr E$, we have
$$
 \frac{(b_{\lambda,N})_t}{\psi'(t)}
 =\frac{\mu}{\beta}(1+b_{\psi,N}^2)
 \ge\frac{\mu}{\beta}.
$$
Using this lower bound and the uniform bounds for
$b_{\lambda, N}$ and $d_xb_{\lambda, N}$ on $\mathscr E$, we can choose
a constant $\Lambda$ independent of $\lambda$ and $u$ such that
\eqref{eq:generalratios} holds along every solution graph
with $q_0=\mu/\beta$.
Applying Proposition~\ref{prop:gradient} on each $U_i$,
with $E=\overline{U_i'}$, and taking the maximum over the
finite cover, we obtain
\begin{equation}\label{eq:psiW}
 \sup_N\bigl(\psi'(u)W_u\bigr)\le C,
 \qquad \text{ where }W_u=\sqrt{1+|du|_{g_N}^2},
\end{equation}
where $C$ is independent of $\lambda$ and $u$.
By \eqref{eq:exist-trapping}, the values of $u$ lie in the
fixed interval $[\min_Nf_-,\max_Nf_+]$.
The function $\psi'$ has a positive minimum on this interval,
so \eqref{eq:psiW} gives a uniform bound for
$\|du\|_{C^0(N,g_N)}$.
By applying Lemma \ref{lem:elliptic} again on the finite cover, we have
\begin{equation}\label{eq:closed-apriori}
 \|u\|_{C^\ell(N)}\le C_\ell
 \qquad(\ell=0,1,2,\ldots).
\end{equation}
The constants depend on $\ell$, $N$, $g_N$, $\eta_N$, $\mu$, $\psi$,
but are independent of $\lambda$ and $u$.

To prove openness, fix $\lambda\in\mathcal I$ and a solution
$u\in\mathcal U$. Define the symmetric $(2,0)$-tensor
$$
 B_u
 =\frac1{W_u}\left(
 g_N^{-1}-\frac{\nabla_{g_N}u\otimes\nabla_{g_N}u}{W_u^2}
 \right)
$$
and the function
$$
 c_u=(b_{\lambda,N})_t(x,u(x))-\frac{\psi''(u)}{W_u}.
$$
For $v\in C^{2,\tau}(N)$, differentiating
\eqref{eq:closed-homotopy} along the variation $u+tv$
at $t=0$, with $\lambda$ fixed, gives
$$
 D\cQ_\lambda(u)v
 =\Div_{g_N,\eta_N}(B_u\,dv)
  -\frac{\psi'(u)}{W_u^3}\ip{du}{dv}_{g_N}-c_uv,
$$
where $B_u\,dv$ denotes the vector field obtained by contracting
$B_u$ with $dv$.
Since
$$
 B_u\,d(\psi(u))
 =\frac{\psi'(u)}{W_u^3}\nabla_{g_N}u,
$$
the linearization can be written as
\begin{equation}\label{eq:linearization}
 D\cQ_\lambda(u)v
 =\Div_{g_N,\eta_N+\psi(u)}(B_u\,dv)-c_uv.
\end{equation}
The tensor $B_u$ is positive definite, and
\eqref{eq:exist-effective-monotonicity} gives
$$
 c_u\ge\frac{\mu}{2\beta}\psi'(u)>0.
$$
By compactness of $N$, $B_u$ is uniformly positive definite and
$c_u$ has a positive minimum.
For $v\in C^\infty(N)$, integration by parts yields
$$
 -\int_N v\,D\cQ_\lambda(u)v\,e^{-\eta_N-\psi(u)}\,dV_{g_N}
 =\int_N\bigl(B_u(dv,dv)+c_uv^2\bigr)
          e^{-\eta_N-\psi(u)}\,dV_{g_N}.
$$
Thus the associated bilinear form is coercive on $H^1(N)$.
The Lax--Milgram theorem and elliptic regularity imply that
$$
 D\cQ_\lambda(u):C^{2,\tau}(N)\longrightarrow C^{0,\tau}(N)
$$
is an isomorphism.
The formula for $b_{\lambda,N}$ is defined for every $\lambda\in\R$,
and $(\lambda,u)\mapsto\cQ_\lambda(u)$ is a smooth map from
$\R\times\mathcal U$ to $C^{0,\tau}(N)$.
Since $D\cQ_\lambda(u)$ is an isomorphism, the implicit function
theorem gives solutions for all parameters in a neighborhood of
$\lambda$. Intersecting this neighborhood with $[0,1]$ shows that
$\mathcal I$ is relatively open in $[0,1]$.

To prove closedness, let $\lambda_j\in\mathcal I$ converge to
$\lambda_\infty\in[0,1]$, and choose corresponding solutions $u_j$.
The estimates \eqref{eq:closed-apriori} give a subsequence of
$\{u_j\}$ converging in $C^\infty(N)$ to a smooth function
$u_\infty$.
Since $f_-\le u_\infty\le f_+$, its graph lies in
$\mathscr E\Subset\mathscr S_{\psi, N}$, and hence
$u_\infty\in\mathcal U$.
Passing to the limit in $\cQ_{\lambda_j}(u_j)=0$ gives
$\cQ_{\lambda_\infty}(u_\infty)=0$.
Hence, $\mathcal I$ is closed in $[0,1]$.

Since $0\in\mathcal I$ and $[0,1]$ is connected,
$\mathcal I=[0,1]$. A solution at $\lambda=1$ satisfies
\eqref{eq:exactpde} and \eqref{eq:exist-strip}.
\end{proof}

We now specialize to $(N,g_N,\eta_N)=(M,h,\eta)$ and write
$\mathcal K=\mathcal K_M$. The associated potential, band, and
prescribed function are $\Phi_\psi$, $\mathscr S_\psi$, and $b_\psi$
from Sections~\ref{sec:scale} and~\ref{sec:gradient};
the norms are now computed using $h$ and $\widehat h$.
Using Proposition~\ref{prop:closed} and an exhaustion argument,
we construct an entire solution with graph contained in $\mathscr S_\psi$.

\begin{proposition}\label{prop:global}
Let $(M,h,\rho,\eta)$ be the fixed background data from
Lemma~\ref{lem:background}, and fix $\mu$ as in
\eqref{eq:mu-choice}.
For every $\psi\in\mathscr P$ satisfying
\eqref{eq:exist-potential-monotonicity}, there exists
$f\in C^\infty(M)$ satisfying
\begin{equation}\label{eq:global-pde}
 \mathcal K(f)=b_\psi(x,f(x))
 \qquad\text{on }M
\end{equation}
and
$$
 \Gamma_f(M)\subset\mathscr S_\psi.
$$
Moreover, every such solution satisfies the local estimate of
Corollary \ref{cor:uniformD}, with a constant independent of
$\psi$ and $f$.
\end{proposition}

\begin{proof}
Fix $\psi$ as in the statement, and choose an exhaustion of $M$
by relatively compact smooth domains $\Omega_j$ such that
$\Omega_j\Subset\Omega_{j+1}$.
For each $j$, choose a collar neighborhood of
$\partial\Omega_{j+1}$ in $\overline{\Omega_{j+1}}$,
identified with $\partial\Omega_{j+1}\times[0,a_j)$ and
disjoint from a neighborhood of $\overline{\Omega_j}$.
Let $z$ denote the collar coordinate, with $z=0$ on the
boundary, and let $k_j$ be the metric induced by $h$ on
$\partial\Omega_{j+1}$, regarded as independent of $z$
on the collar.

Choose $\chi_j\in C^\infty(\overline{\Omega_{j+1}})$ with
$0\le\chi_j\le1$, supported in the collar and equal to one
near $\partial\Omega_{j+1}$.
On the collar, set
$$
 \widetilde h_j
 =(1-\chi_j)h+\chi_j(dz^2+k_j),\qquad
 \widetilde\rho_j=(1-\chi_j)\rho.
$$
Outside the collar, set $\widetilde h_j=h$ and
$\widetilde\rho_j=\rho$.
These definitions give smooth data on
$\overline{\Omega_{j+1}}$.
The metric $\widetilde h_j$ is positive definite.
The modified data agree with $h$, $\rho$ on a neighborhood of
$\overline{\Omega_j}$, while near $\partial\Omega_{j+1}$
we have $\widetilde h_j=dz^2+k_j$ and
$\widetilde\rho_j=0$.

Doubling $\overline{\Omega_{j+1}}$ across its boundary gives
a closed manifold $N_j$ on which the modified metric and
function extend smoothly. Denote these extensions by
$h_j$, $\rho_j$, and set $\eta_j=-\beta\rho_j$.
On the copy of $\Omega_j$ in $N_j$, the data
$h_j$, $\rho_j$, $\eta_j$ agree with $h$, $\rho$, $\eta$.
Apply Proposition \ref{prop:closed} with
$(N,g_N,\eta_N)=(N_j,h_j,\eta_j)$ and the same $\mu$ and $\psi$.
The proposition requires only smoothness of the background data
and compactness of $N_j$. Let $f_j$ be the restriction of the resulting
solution to $\Omega_j$.
On this copy, the auxiliary operator and prescribed function agree
with $\mathcal K$ and $b_\psi$ for the original background. Thus
$$
 \mathcal K(f_j)=b_\psi(x,f_j(x))
 \qquad\text{on }\Omega_j,
$$
and
$$
 \Gamma_{f_j}(\Omega_j)\subset\mathscr S_\psi.
$$

Fix a nonempty compact set $E\subset M$, and choose relatively
compact smooth domains with
$$
 E\Subset U_1\Subset U_0\Subset M.
$$
Choose $j_0$ such that $\overline{U_0}\subset\Omega_j$
for every $j\ge j_0$. All estimates below concern these indices.
Since $\eta=-\beta\rho$ and $\psi$ is increasing, the band
condition gives $f_j(\overline{U_0})\subset I_{U_0}$, where
$$
 I_{U_0}=
 \left[
 \psi^{-1}\left(
   \beta\min_{\overline{U_0}}\rho-\frac{\beta\pi}{2\mu}
 \right),
 \psi^{-1}\left(
   \beta\max_{\overline{U_0}}\rho+\frac{\beta\pi}{2\mu}
 \right)
 \right].
$$

Applying Corollary \ref{cor:uniformD} on $U_0$ with compact
subset $\overline{U_1}$, we have
\begin{equation}\label{eq:cU1U0}
 \sup_{\overline{U_1}}
 \left(\psi'(f_j)\sqrt{1+|df_j|_h^2}\right)
 \le C_{\overline{U_1},U_0},
\end{equation}
where $C_{\overline{U_1},U_0}$ is independent of $j$ and $\psi$.
Set
$$
 c_{U_0}=\min_{t\in I_{U_0}}\psi'(t)>0.
$$
Combining \eqref{eq:cU1U0} with this lower bound for
$\psi'(f_j)$ yields
\begin{equation}\label{eq:f_jbound}
 \|df_j\|_{C^0(\overline{U_1},h)}
 \le\frac{C_{\overline{U_1},U_0}}{c_{U_0}}.
\end{equation}
The resulting gradient bound may depend on $\psi$, but is
independent of $j$.

We next construct local barriers to place the graphs over
$U_1$ in a fixed compact subset of $\mathscr S_\psi$, as
required by Lemma \ref{lem:elliptic}.
Choose $\zeta\in C^\infty(\overline{U_0})$ with
$\zeta=0$ on $\partial U_0$, $0<\zeta\le1$ in $U_0$,
and $\zeta=1$ near $\overline{U_1}$.
For $0<\varepsilon\le\pi/4$, define
\begin{equation}\label{eq:localbarriers}
 f^\pm_\varepsilon
 =\psi^{-1}\left(
   \beta\rho\pm\frac{\beta}{\mu}
                  (\frac\pi 2-\varepsilon\zeta)
 \right).
\end{equation}
On $\partial U_0$, the functions $f^-_\varepsilon$ and
$f^+_\varepsilon$ equal the lower and upper endpoints of
the vertical interval in $\mathscr S_\psi$, respectively.
Hence the band condition for $f_j$ gives
\begin{equation}\label{eq:fjfeps}
 f^-_\varepsilon<f_j<f^+_\varepsilon
 \qquad\text{on }\partial U_0.
\end{equation}

Note the values of $f^\pm_\varepsilon$ lie in $I_{U_0}$.
Since $\psi(I_{U_0})$ is compact, $\psi^{-1}$ and its first
two derivatives are bounded on  $\psi(I_{U_0})$.
The functions $f^\pm_\varepsilon$ defined by
\eqref{eq:localbarriers} are therefore uniformly bounded
in $C^2(\overline{U_0})$ for $0<\varepsilon\le\pi/4$.
By \eqref{eq:exist-curvature-operator}, there is a constant
$C_{\mathrm{bar},U_0}<\infty$, independent of $j$ and
$\varepsilon$, such that
$$
 \max_{\sigma\in\{-,+\}}
 \|\mathcal K(f^\sigma_\varepsilon)\|_{C^0(\overline{U_0})}
 \le C_{\mathrm{bar},U_0}.
$$
In $U_0$, the barrier graphs lie in $\mathscr S_\psi$, and
$$
 b_\psi(x,f^+_\varepsilon(x))
 =\cot(\varepsilon\zeta(x))\ge\cot\varepsilon,
$$
whereas
$$
 b_\psi(x,f^-_\varepsilon(x))
 =-\cot(\varepsilon\zeta(x))\le-\cot\varepsilon.
$$
Choose and fix $\varepsilon\in(0,\pi/4]$ such that
$\cot\varepsilon>C_{\mathrm{bar},U_0}+1$.
It follows that
$$
 \mathcal K(f^-_\varepsilon)
 -b_\psi(x,f^-_\varepsilon(x))>0
 \qquad\text{on }U_0,
$$
and
\begin{equation}\label{eq:K-b}
 \mathcal K(f^+_\varepsilon)
 -b_\psi(x,f^+_\varepsilon(x))<0
 \qquad\text{on }U_0.
\end{equation}

Suppose that $f_j\ge f^+_\varepsilon$ at some point of $U_0$.
The strict boundary inequality \eqref{eq:fjfeps} implies that
$f_j-f^+_\varepsilon$ attains a nonnegative maximum at an
interior point $x_+\in U_0$, where
$df_j=df^+_\varepsilon$ and
$\Hess_hf_j\le\Hess_hf^+_\varepsilon$.
The principal coefficient matrix of $\mathcal K$ is positive
definite and depends only on the base point and the gradient.
Its coefficients therefore agree for $f_j$ and
$f^+_\varepsilon$ at $x_+$.
By the Hessian inequality and
\eqref{eq:exist-effective-monotonicity}, we obtain
$$
 \begin{aligned}
 0
 &=\mathcal K(f_j)(x_+)-b_\psi(x_+,f_j(x_+))\\
 &\le\mathcal K(f^+_\varepsilon)(x_+)
       -b_\psi(x_+,f^+_\varepsilon(x_+))<0,
 \end{aligned}
$$
where the strict inequality follows from \eqref{eq:K-b}.
This contradiction proves that $f_j<f^+_\varepsilon$ on $U_0$.
Applying the same argument at a hypothetical nonpositive minimum
of $f_j-f^-_\varepsilon$ gives $f_j>f^-_\varepsilon$ on $U_0$.
Consequently,
$$
 f^-_\varepsilon<f_j<f^+_\varepsilon
 \qquad\text{on }U_0.
$$

Since $\psi$ is increasing and $\zeta=1$ near
$\overline{U_1}$, \eqref{eq:localbarriers} gives
\begin{equation}\label{eq:separation}
 \left|\frac{\mu}{\beta}\Phi_\psi(x,f_j(x))\right|
 \le\frac{\pi}{2}-\varepsilon
 \qquad\text{on }\overline{U_1}.
\end{equation}
Here $\varepsilon$ is independent of $j$, but may depend on
the fixed $\psi$.

By \eqref{eq:separation} and $f_j(\overline{U_0})\subset I_{U_0}$, all graphs
$\Gamma_{f_j}(U_1)$ lie in the compact set
$$
 \mathscr E=
 \left\{(x,t)\in\overline{U_1}\times I_{U_0}:
 \left|\frac{\mu}{\beta}\Phi_\psi(x,t)\right|
 \le\frac{\pi}{2}-\varepsilon\right\}
 \Subset\mathscr S_\psi.
$$
Choose an open neighborhood $\mathscr O_0$ of $\mathscr E$
with $\mathscr O_0\Subset\mathscr S_\psi$.
The smoothness of $b_\psi$ on $\mathscr S_\psi$ gives
$\|b_\psi\|_{C^k(\mathscr O_0)}<\infty$ for every $k\ge0$.

Choose finitely many coordinate domains
$V_i'\Subset V_i\Subset U_1$ such that $E\subset\bigcup_iV_i'$.
The inclusion $\Gamma_{f_j}(V_i)\subset\mathscr E$,
the uniform gradient bound \eqref{eq:f_jbound}, and the bounds for all derivatives of $b_\psi$ on $\mathscr O_0$
allow us to apply Lemma~\ref{lem:elliptic} to the restrictions
$f_j|_{V_i}$, using the same function $b_\psi(x,t)$
for every $j$.
Combining the resulting estimates on $V_i'$ over the
finite cover of $E$, we obtain, for every integer $\ell\ge0$,
$$
 \sup_{j\ge j_0}\|f_j\|_{C^\ell(E)}\le C_{\ell,E,\psi}<\infty.
$$
These constants may depend on $\psi$ and the chosen domains,
but are independent of $j$.

By the Arzel\`a--Ascoli theorem and a diagonal argument along
the exhaustion $\{\Omega_j\}$, a subsequence of $\{f_j\}$ converges
in $C^\infty_{\mathrm{loc}}(M)$ to a smooth function $f$.
For each compact set $E$, the separation estimate
\eqref{eq:separation} passes to the limit and places
$\Gamma_f(E)$ inside $\mathscr S_\psi$.
Thus $\Gamma_f(M)\subset\mathscr S_\psi$.
The smooth local convergence and the smoothness of $b_\psi$
near these graphs allow passage to the limit in the equation
for $f_j$, which yields \eqref{eq:global-pde}.
Finally, by applying Corollary \ref{cor:uniformD} to every smooth
solution of \eqref{eq:global-pde} whose graph lies in
$\mathscr S_\psi$, we obtain the asserted local estimate.
\end{proof}

\section{Proof of the main theorem}\label{sec:choice}

Throughout this section, we retain the fixed background data
$(M,h,\rho,\eta)$ and $\delta>0$ from
Lemma \ref{lem:background}, and the choice of $\mu$ in
\eqref{eq:mu-choice}. Recall that $\beta=n-1$.
In this section, we prove Theorem~\ref{thm:main} by choosing the
vertical potential $\psi$ and applying Proposition~\ref{prop:global}
to construct an entire graph. We show that the conformal metric
$\widetilde g=e^{-2\theta/(n-1)}G$, with $\theta=\varphi+\log W$,
is complete and has uniformly positive scalar curvature.

We first use Corollary~\ref{cor:uniformD} to bound $\cD$
by a smooth function of $\psi(f)$ chosen independently
of $\psi$ and $f$.

\begin{lemma}\label{lem:majorant}
There is a smooth function $\cB:\R\to[1,\infty)$, depending
only on the fixed background data $(M,h,\rho,\eta)$ and $\mu$,
such that, for every $\psi\in\mathscr P$ and every
$f\in C^\infty(M)$ satisfying \eqref{eq:global-pde} with
$\Gamma_f(M)\subset\mathscr S_\psi$, we have
\begin{equation}\label{eq:majorant}
 \cD=\psi'(f)W\le\cB(\psi(f))
 \qquad\text{on }M.
\end{equation}
\end{lemma}

\begin{proof}
For each $j\in\mathbb Z$, set
$$
 E_j=
 \left\{x\in M:
 \rho(x)\le\frac{j+1}{\beta}+\frac{\pi}{2\mu}\right\}.
$$
Each $E_j$ is compact because $\rho$ is proper and bounded
below. If $\psi(f(x))\in[j,j+1]$, by the band condition, we have
$$
 \rho(x)
 <\frac{\psi(f(x))}{\beta}+\frac{\pi}{2\mu}
 \le\frac{j+1}{\beta}+\frac{\pi}{2\mu},
$$
and hence $x\in E_j$.
For every nonempty $E_j$, choose an open set
$U_j\Subset M$ with $E_j\Subset U_j$.
By Corollary~\ref{cor:uniformD}, there is a constant
$B_j\ge1$, independent of $\psi$ and $f$, such that
$\cD\le B_j$ on $E_j$.
Set $B_j=1$ when $E_j$ is empty.

Choose a nonnegative, locally finite smooth partition of
unity $\{\chi_i\}$ on $\R$, with compact supports.
For each $i$, the compactness of $\supp\chi_i$ implies that
only finitely many integers $j$ satisfy
$[j,j+1]\cap\supp\chi_i\ne\varnothing$.
Set
$$
 d_i=\max\left(
 \{1\}\cup
 \left\{B_j:j\in\mathbb Z,\,
 [j,j+1]\cap\supp\chi_i\ne\varnothing\right\}
 \right),
$$
and define
$$
 \cB(y)=\sum_i d_i\chi_i(y).
$$
Then $\cB$ is smooth and $\cB\ge1$.
If $y\in[j,j+1]$ and $\chi_i(y)>0$, the choice of $d_i$
gives $d_i\ge B_j$. Hence
$$
 \cB(y)\ge B_j\sum_i\chi_i(y)=B_j.
$$
For any $x\in M$, choose $j\in\mathbb Z$ such that
$\psi(f(x))\in[j,j+1]$. Then $x\in E_j$, and hence
$$
 \cD(x)\le B_j\le\cB(\psi(f(x))),
$$
which proves \eqref{eq:majorant}.
By construction, $\cB$ depends only on the fixed background
data $(M,h,\rho,\eta)$ and $\mu$, and is independent of
$\psi$ and $f$.
\end{proof}

We next construct a potential whose derivatives satisfy
prescribed bounds in terms of its value.

\begin{lemma}\label{lem:small}
Let $\varepsilon_*:\R\to(0,1]$ be a continuous function.
There exists $\psi\in\mathscr P$ such that, for every
$t\in\R$,
\begin{equation}\label{eq:smallfunction}
 0<\psi'(t)\le\varepsilon_*(\psi(t)),
\end{equation}
\begin{equation}\label{eq:small-second-derivative}
 |\psi''(t)|\le\varepsilon_*(\psi(t))^2,
\end{equation}
\begin{equation}\label{eq:small-log-derivative}
 |(\log\psi')'(t)|\le\varepsilon_*(\psi(t)).
\end{equation}
\end{lemma}

\begin{proof}
Choose a nonnegative, locally finite smooth partition of unity
$\{\phi_j\}$ on $\R$ with compact supports, such that
each point belongs to at most $L$ supports for some
$L<\infty$.
Choose constants $c_j>0$ satisfying
$$
 c_j\max\left\{
 \|\phi_j\|_\infty,\,
 \|\phi_j'\|_\infty,\,
 \|\phi_j''\|_\infty
 \right\}
 \le\frac1L\min_{y\in\supp\phi_j}\varepsilon_*(y).
$$
The minima on the right are positive by continuity of
$\varepsilon_*$ and compactness of the supports.
Set
\[
 \gamma(y)=\sum_j c_j\phi_j(y).
\]
This locally finite sum defines a smooth positive function.
Since each point $y$ belongs to at most $L$ of the sets
$\supp\phi_j$, the choice of $c_j$ gives
\[
 0<\gamma(y)\le\varepsilon_*(y),
\]
and
$$
 \max\{|\gamma'(y)|, |\gamma''(y)|\}\le\varepsilon_*(y).
$$

Define
$$
 T(y)=\int_0^y\frac{dz}{\gamma(z)}.
$$
Since $0<\gamma\le1$, we have $T'\ge1$ and
$T(y)\to\pm\infty$ as $y\to\pm\infty$.
Thus $T$ is a smooth increasing diffeomorphism of $\R$.
Let $\psi=T^{-1}$.
Then $\psi(0)=0$, and
$$
 \psi'(t)=\gamma(\psi(t)).
$$
It follows that
$$
 \psi''(t)=\gamma(\psi(t))\gamma'(\psi(t)),
\qquad
 (\log\psi')'(t)=\gamma'(\psi(t)),
\qquad
 (\log\psi')''(t)
 =\gamma(\psi(t))\gamma''(\psi(t)).
$$
By the bounds for $\gamma$, $\gamma'$ and $\gamma''$, we obtain
\eqref{eq:smallfunction},
\eqref{eq:small-second-derivative} and
\eqref{eq:small-log-derivative}.
Since $\varepsilon_*\le1$, these also imply
$0<\psi'\le1$, $|(\log\psi')'|\le1$, and
$$
 (\log\psi')''(t)
 \ge-\varepsilon_*(\psi(t))^2\ge-1.
$$
Thus $\psi\in\mathscr P$.
\end{proof}

We next give sufficient conditions for a metric with
uniformly positive weighted scalar curvature to admit
a complete conformal metric with uniformly positive
scalar curvature.

\begin{lemma}\label{lem:remove-weight}
Let $(N,g_N)$ be a connected smooth Riemannian manifold
of dimension $n\ge3$ without boundary.
Let $\theta_N,\,u_N\in C^\infty(N)$, and assume that $u_N$
is proper.
Suppose that, for some constants $\lambda>0$ and $B\ge0$,
$$
 \mathcal R_n(g_N,\theta_N)\ge\lambda,
 \qquad \theta_N\ge-B
 \qquad\text{on }N,
$$
and
$$
 e^{\frac{2\theta_N}{n-1}}|du_N|_{g_N}^2\le1
 \qquad\text{on }N.
$$
Then the metric
$$
 \widetilde g_N=e^{-\frac{2\theta_N}{n-1}}g_N
$$
is complete and satisfies
$$
 R_{\widetilde g_N}\ge\lambda e^{-\frac{2B}{n-1}}
 \qquad\text{on }N.
$$
\end{lemma}

\begin{proof}
By the conformal scalar curvature formula and
\eqref{eq:critical-weighted-scalar}, we have
$$
 \begin{aligned}
 R_{\widetilde g_N}
 =e^{\frac{2\theta_N}{n-1}}
   \left(R_{g_N}+2\Delta_{g_N}\theta_N
         -\frac{n-2}{n-1}|d\theta_N|_{g_N}^2\right)&=e^{\frac{2\theta_N}{n-1}}\mathcal R_n(g_N,\theta_N)\\
 &\ge\lambda e^{-\frac{2B}{n-1}}.
 \end{aligned}
$$

To prove completeness, observe that
$$
 |du_N|_{\widetilde g_N}^2
 =e^{\frac{2\theta_N}{n-1}}|du_N|_{g_N}^2\le1.
$$
For any $x,y\in N$, integrating $du_N$ along piecewise smooth
curves from $x$ to $y$ and taking the infimum of their
lengths gives
$$
 |u_N(x)-u_N(y)|\le d_{\widetilde g_N}(x,y).
$$
For every $x\in N$ and $r>0$, we therefore have
$$
 \{y\in N:d_{\widetilde g_N}(x,y)\le r\}
 \subset u_N^{-1}([u_N(x)-r,u_N(x)+r]).
$$
The inverse image $u_N^{-1}([u_N(x)-r,u_N(x)+r])$ is compact
because $u_N$ is proper.
Since $d_{\widetilde g_N}$ induces the manifold topology,
the closed ball is compact as a closed subset of
$u_N^{-1}([u_N(x)-r,u_N(x)+r])$.
As $x\in N$ and $r>0$ are arbitrary, the Hopf--Rinow
theorem implies that $(N,\widetilde g_N)$ is complete.
\end{proof}

Now we are ready to prove the main theorem.
\begin{proof}[Proof of Theorem \ref{thm:main}]
Let $\cB$ be the function given by
Lemma \ref{lem:majorant}, and set
$$
 \varepsilon_0
 =\min\left\{
 1,\frac{\sqrt\delta}{4},\frac{\mu}{2\beta}
 \right\}.
$$
Define
\begin{equation}\label{eq:choice-smallness}
 \varepsilon_*(y)
 =\min\left\{
 \varepsilon_0,\,
 e^{-\frac{\pi\beta}{2\mu(\beta-1)}}
       \cB(y)^{-\frac{1}{\beta-1}}
 \right\}.
\end{equation}
Since $\beta-1=n-2>0$ and $\cB\ge1$, the function
$\varepsilon_*:\R\to(0,1]$ is continuous.
Its definition depends only on the fixed background data
$(M,h,\rho,\eta)$, $\delta$, and $\mu$.
Apply Lemma~\ref{lem:small} to $\varepsilon_*$, and fix
the resulting function $\psi\in\mathscr P$.
By \eqref{eq:smallfunction} and
\eqref{eq:small-second-derivative},
$$
 0<\psi'(t)\le\varepsilon_0,\qquad
 |\psi''(t)|\le\varepsilon_0^2
 \qquad\text{ for every }t\in\R.
$$
Moreover, \eqref{eq:small-log-derivative} gives
$$
 |(\log\psi')'(t)|
 \le\varepsilon_*(\psi(t))
 \le\varepsilon_0
 \le\frac{\mu}{2\beta}
 \qquad\text{ for every } t\in\R.
$$

For this choice of $\psi$, by Lemma \ref{lem:background}
and \eqref{eq:ambient-weight-change}, we have
$$
 \mathcal R_{n+1}(\widehat h,\Phi_\psi)
 \ge\delta-\left(2+\frac{\beta}{n}\right)\varepsilon_0^2
 \ge\frac{3\delta}{4}
 \qquad\text{on }M\times\R,
$$
where we used $\beta/n<1$ and
$\varepsilon_0^2\le\delta/16$.
The bound for $|(\log\psi')'|$ verifies
\eqref{eq:exist-potential-monotonicity}.
Proposition~\ref{prop:global} therefore provides
$f\in C^\infty(M)$ satisfying 
$$
 \mathcal K(f)=b_\psi(x,f(x))
 \qquad\text{on }M,
$$
and $\Gamma_f(M)\subset\mathscr S_\psi$.

For this graph, write
$$
 G=h+df\otimes df,\qquad W=\sqrt{1+|df|_h^2},
$$
and
$$
 \varphi=\psi(f)-\beta\rho,\qquad
 \theta=\varphi+\log W.
$$
Define the conformal metric
$$
 \widetilde g=e^{-\frac{2\theta}{\beta}}G.
$$
By Proposition \ref{prop:scalar}, we obtain
\begin{equation}\label{eq:weighted-uniformpositive}
 \begin{aligned}
 \mathcal R_n(G,\theta)
 &\ge\frac{3\delta}{4}-C_n\psi'(f)^2\\
 &\ge\frac{3\delta}{4}-C_n\varepsilon_0^2
 \ge\frac{\delta}{2}
 \qquad\text{on }M,
 \end{aligned}
\end{equation}
where we use $C_n=\frac{n^2+1}{n(n+1)}<1$ and
$\varepsilon_0^2\le\delta/16$.
By \eqref{eq:normalized-weight-lower}, we also have
\begin{equation}\label{eq:thetalow}
 \theta\ge-\frac{\beta\pi}{2\mu}
 \qquad\text{on }M.
\end{equation}

Set $u=\psi\circ f$.
Since $\varphi=u-\beta\rho$, the band condition $\Gamma_f(M)\subset\mathscr S_\psi$ gives
\begin{equation}\label{eq:properheight}
 u>\beta\rho-\frac{\beta\pi}{2\mu}
 \qquad\text{on }M.
\end{equation}
For every compact interval $[t_1,t_2]\subset\R$, it follows that
$$
 u^{-1}([t_1,t_2])
 \subset
 \left\{x\in M:
 \rho(x)\le\frac{t_2}{\beta}+\frac{\pi}{2\mu}\right\}.
$$
By Lemma \ref{lem:background}, $\rho$ is proper and
nonnegative, so
$\{\rho\le t_2/\beta+\pi/(2\mu)\}$ is compact.
Its closed subset $u^{-1}([t_1, t_2])$ is therefore compact.
Thus $u$ is proper.

For $\cD=\psi'(f)W$, by Lemma \ref{lem:majorant} we know
$\cD\le\cB(u)$ on $M$.
Using $du=\psi'(f)df$, \eqref{eq:AV}, and the band bound
$\varphi<\beta\pi/(2\mu)$, we obtain
\begin{equation}\label{eq:propergradient}
 \begin{aligned}
 |du|_{\widetilde g}^2
 =e^{\frac{2\theta}{\beta}}|du|_G^2
 &=e^{\frac{2\varphi}{\beta}}W^{\frac{2}{\beta}}
       \psi'(f)^2(1-W^{-2})\\
 &\le e^{\frac{\pi}{\mu}}
       \psi'(f)^{\frac{2(\beta-1)}{\beta}}\cD^{\frac{2}{\beta}}\\
 &\le e^{\frac{\pi}{\mu}}
       \psi'(f)^{\frac{2(\beta-1)}{\beta}}\cB(u)^{\frac{2}{\beta}}.
 \end{aligned}
\end{equation}
Using \eqref{eq:smallfunction} and
\eqref{eq:choice-smallness} in \eqref{eq:propergradient},
we conclude that
\begin{equation}\label{eq:dule1}
 |du|_{\widetilde g}^2
 \le e^{\frac{\pi}{\mu}}
       \varepsilon_*(u)^{\frac{2(\beta-1)}{\beta}}\cB(u)^{\frac{2}{\beta}}
 \le1.
\end{equation}

Since $u$ is proper, \eqref{eq:weighted-uniformpositive}, \eqref{eq:thetalow} and \eqref{eq:dule1} allow us to apply Lemma \ref{lem:remove-weight} with
$(N,g_N)=(M,G)$, $\theta_N=\theta$, $u_N=u$,
$\lambda=\delta/2$, and $B=\beta\pi/(2\mu)$.
Since $\beta=n-1$, the conformal metric in that lemma
is precisely $\widetilde g$.
It follows that $\widetilde g$ is complete and
$$
 R_{\widetilde g}
 \ge\frac{\delta}{2}e^{-\frac\pi\mu}>0.
$$
Set $\sigma=\frac{\delta}{2}e^{-\frac\pi\mu}$ and define
\begin{equation}\label{eq:finalmetric}
 \bar g
 =\sigma\widetilde g.
\end{equation}
The metric $\bar g =\sigma e^{-\frac{2\theta}{\beta}}(h+df\otimes df)$ is smooth and complete. Moreover,
$$
 R_{\bar g}
 =\sigma^{-1}R_{\widetilde g}\ge1
 \qquad\text{on }M.
$$
This proves Theorem \ref{thm:main}.
\end{proof}

\end{document}